\documentclass[a4paper,12pt,reqno]{amsart}

\usepackage{amsmath,amssymb,amsthm}
\usepackage{latexsym}
\usepackage{color}
\usepackage{graphicx}
\usepackage{mathrsfs}
\usepackage{enumerate}
\usepackage[abbrev]{amsrefs}
\usepackage[T1]{fontenc}
\allowdisplaybreaks[4]

\theoremstyle{plain}
\newtheorem{thm}{Theorem}[section]
\newtheorem{lemm}[thm]{Lemma}

\theoremstyle{definition}

\newtheorem{rem}[thm]{Remark}

\makeatletter

\@addtoreset{equation}{section}
\makeatother

\begin{document}
\title[Time periodic solutions to the $2$D quasi-geostrophic equation]
{Time periodic solutions to the $2$D quasi-geostrophic equation with the supercritical dissipation}
\author[Mikihiro Fujii]{Mikihiro Fujii}
\address[]{Graduate School of Mathematics Kyushu University,Fukuoka 819--0395, JAPAN}
\email{3MA20005M@s.kyushu-u.ac.jp}
\begin{abstract}
	We consider the $2$D dissipative quasi-geostrophic equation with the time periodic external force
	and prove the existence of a unique time periodic solution in the case of the supercritical dissipation.
	In this case, the smoothing effect of the semigroup generated by the dissipation term is too weak to control the nonlinearity in the Duhamel term of the correponding integral equation.
	In this paper, we give a new approach which does not depend on the contraction mapping principle for the integral equation.
\end{abstract}
\maketitle

\section{Introduction}\label{s1}
We consider the $2$D dissipative quasi-geostrophic equation with the time periodic external force:
\begin{equation}\label{QG}
	\begin{cases}
		\partial_t\theta
			+(-\Delta)^{\frac{\alpha}{2}}\theta
			+u \cdot \nabla \theta
			=F,
			&\qquad t>0, x\in \mathbb{R}^2,\\
		u=\mathcal{R}^{\perp}\theta=(-\mathcal{R}_2\theta,\mathcal{R}_1\theta),
			&\qquad t\geqslant 0, x\in \mathbb{R}^2,\\
		\theta(0,x)=\theta_0(x),
			&\qquad x\in \mathbb{R}^2,
	\end{cases}
\end{equation}
where $\theta=\theta(t,x)$ and $u=(u_1(t,x),u_2(t,x))$ represent the unknown potential temperature of the fluid with some initial value $\theta_0$ and the unknown velocity field of the fluid, respectively.
The given external force $F=F(t,x)$ is $T$-time periodic, that is $F$ satisfies $F(t+T)=F(t)$ $(t>0)$ for some $T>0$.
The two operators $(-\Delta)^{\frac{\alpha}{2}}$ $(0<\alpha\leqslant 2)$ and $\mathcal{R}_k$ $(k=1,2)$
denote the nonlocal differential operators
so-called the fractional Laplacian and the Riesz transforms on $\mathbb{R}^2$, respectively and they are defined by
\begin{equation*}\label{1.1}
		(-\Delta)^{\frac{\alpha}{2}}f
			=\mathscr{F}^{-1}\left[|\xi|^{\alpha} \widehat{f}(\xi) \right],
		\qquad
		\mathcal{R}_kf
			=\partial_{x_k}(-\Delta)^{-\frac{1}{2}}f
			=\mathscr{F}^{-1}\left[\frac{i\xi_k}{|\xi|} \widehat{f}(\xi) \right].
\end{equation*}

In this paper, we prove the existence of a unique $T$-time periodic solution of (\ref{QG}) with the supercritical dissipation
if the given $T$-time periodic external force is sufficiently small.

Before we state the main result precisely,
we recall some known results for the $2$D dissipative quasi-geostrophic equation
with the case $F=0$, that is the usual initial value problem:
\begin{equation}\label{IVP_QG}
	\begin{cases}
		\partial_t\theta+(-\Delta)^{\frac{\alpha}{2}}\theta+u\cdot \nabla \theta=0,
			\qquad & t>0,x\in \mathbb{R}^2,\\
		u=\mathcal{R}^{\perp}\theta=(-\mathcal{R}_2\theta,\mathcal{R}_1\theta),
		 	\qquad & t\geqslant 0,x\in \mathbb{R}^2,\\
		\theta(0,x)=\theta_0(x),
			\qquad & x\in \mathbb{R}^2.
	\end{cases}
\end{equation}
Based on the scaling transform and the $L^{\infty}(\mathbb{R}^2)$-conservation, the dissipative quasi-geostrophic equation is divided
into the subcritical case $1<\alpha\leqslant 2$, critical case $\alpha=1$ and supercritical case $0<\alpha<1$.
In the subcritical case, Constantin-Wu \cite{CW} proved the existence of a weak solution and decay estimates with respect to $L^2$ norm for the initial data $\theta_0\in L^2(\mathbb{R}^2)$.
Wu \cite{Wu2} proved the global well-posedness for small data in the scaling subcritical setting $\theta_0\in L^p(\mathbb{R}^2)$ ($p>2/(\alpha-1)$) via the contraction mapping principle for the correponding integral equation.
In the critical case, the order of the spatial derivative in the dissipation term coincides with that in the nonlinear term.
Zhang \cite{Zhang} noticed this property and proved the existence of the global in time mild solution in the scaling critical Besov space $\dot{B}_{p,1}^{\frac{2}{p}}(\mathbb{R}^2)$ ($1\leqslant p \leqslant \infty$).
Global well-posedness in the Tribel-Lizorkin spaces $F_{p,q}^s(\mathbb{R}^2)$ ($s>2/p$, $1<p,q<\infty$) is proved by Chen-Zhang \cite{CZ}.
In the supercritical case, the order of the spatial derivative in the dissipation term is less than that in the nonlinear term.
Therefore,
the smoothing effect of the fractional heat kernel $e^{-t(-\Delta)^{\frac{\alpha}{2}}}$ is too weak to control the spatial derivative in the nonlinear term.
Although, this implies that it seems to be impossible to construct a solution of (\ref{IVP_QG}),
it is able to overcome this and the local well-posedness for large data and the global well-posedness for small data
in the scaling critical Sobolev $H^{2-\alpha}(\mathbb{R}^2)$ by Miura \cite{Miura}
and Besov spaces ${B}_{p,q}^{1+\frac{2}{p}-\alpha}(\mathbb{R}^2)$ ($2\leqslant p<\infty$, $1\leqslant q<\infty$)
by Bae \cite{Bae}, Chae-Lee \cite{CL} and Chen-Miao-Zhang \cite{CMZ2}.
Their method is based on the energy estimates for the iteration of the transport-diffusion type equation and they control the nolinear term
by the divergence free condition $\nabla \cdot u=0$ and the commutator estimates.

On the other hand, despite the large number of previous studies on the well-posedness of the initial value problem (\ref{IVP_QG}),
the study on the existence of time periodic solutions to the $2$D quasi-geostrophic equation  is hardly known.

In this manuscript, we consider the supercritical case and prove
the existence of a unique suitable initial data and a unique time periodic solution to (\ref{QG}) in the scaling critical Besov space
if the given time periodic external force is sufficiently small.
More precisely, our main result of this paper reads as follows.
\begin{thm}\label{t1-1}
	Let $T>0$ and $2/3<\alpha<1$.
	Let exponents $p$, $q$ and $r$ satisfy
	\begin{equation}\label{1.2}
		\frac{2}{2\alpha-1}<r\leqslant p<\frac{4}{\alpha},
		\qquad
		1\leqslant q<\infty.
	\end{equation}
	Then, there exist positive constants $\delta=\delta(\alpha,p,q,r,T)$ and $K=K(\alpha,p,q,r,T)$
	such that
	if the given $T$-time periodic external force $F\in BC((0,\infty); \dot{B}_{r,\infty}^0(\mathbb{R}^2))$ satisfies
	\begin{equation*}\label{1.3}
		\sup_{t>0} \| F(t) \|_{\dot{B}_{r,\infty}^0}\leqslant \delta,
	\end{equation*}
	then there exist a unique initial data $\theta_0 \in B_{p,q}^{1+\frac{2}{p}-\alpha}(\mathbb
	{R}^2)$ and a unique $T$-time periodic solution $\theta$ to (\ref{QG}) satisfying
	\begin{equation}\label{1.4}
		\theta \in BC([0,\infty);B_{p,q}^{1+\frac{2}{p}-\alpha}(\mathbb{R}^2)),
		\qquad
		\|\theta\|_{\widetilde{L}^{\infty}(0,\infty;\dot{B}_{p,q}^{1+\frac{2}{p}-\alpha})}\leqslant K.
	\end{equation}
\end{thm}
\begin{rem}\label{t1-2}

	\noindent
	\begin{itemize}
		\item [(1)]
			If $\theta$ and $F$ satisfy (\ref{QG}), then
			\begin{equation*}\label{1.5}
				\theta_{\lambda}(t,x)=\lambda^{\alpha-1}\theta(\lambda^{\alpha}t,\lambda x),\quad
				F_{\lambda}(t,x)=\lambda^{2\alpha-1}F(\lambda^{\alpha}t,\lambda x)
			\end{equation*}
			also satisfy (\ref{QG}) for all $\lambda>0$.
			Since it holds
			\begin{equation*}\label{1.6}
				\begin{split}
					\sup_{t\geqslant 0}\|\theta_{\lambda}(t)\|_{\dot{B}_{p,q}^{1+\frac{2}{p}-\alpha}}
						&=\sup_{t\geqslant 0}\|\theta(t)\|_{\dot{B}_{p,q}^{1+\frac{2}{p}-\alpha}},\\
					\sup_{t>0}\|F_{\lambda}(t)\|_{\dot{B}_{2/(2\alpha-1),\infty}^0}
						&=\sup_{t>0}\|F(t)\|_{\dot{B}_{2/(2\alpha-1),\infty}^0}
				\end{split}
			\end{equation*}
			for all dyadic numbers $\lambda>0$,
			the function spaces $BC([0,\infty);B_{p,q}^{1+\frac{2}{p}-\alpha}(\mathbb{R}^2))$ and $BC((0,\infty);\dot{B}_{r,\infty}^0(\mathbb{R}^2))$ in Theorem \ref{t1-1} are scaling critical and subcritical setting, respectively.
		\item [(2)]
			Our smallness condition $\delta$ and $K$ depend continuously on $T$ and go to $0$ as $T\to +0$ or $T\to \infty$.
			Hence, we can take $\delta$ and $K$ local uniformly for $T\in (0,\infty)$.
		\item [(3)]
			The assumption $2/3<\alpha$ in Theorem \ref{t1-1} ensures the exsitence of $p$ and $r$ satisfying (\ref{1.2}).
	\end{itemize}
\end{rem}
In the case of the Navier-Stokes equation, the existence of time periodic solutions is often proved by applying the contraction mapping principle to the correponding integral equation.
(It was in \cite{KN} that first used this idea.)
However, in the case of our problem, the supercritical dissipation prevents us from using this scheme.
Indeed, when we apply the idea of \cite{KN} to (\ref{QG}) on the whole time line $\mathbb{R}$,
we meet the difficulty that
the smoothing effect of the fractional heat kernel $e^{-t(-\Delta)^{\frac{\alpha}{2}}}$ is too weak to control the first order spatial derivative of the nonlinear term and it is pretty difficult to find a Banach space $X$ satisfying
\begin{equation*}\label{1.7}
	\sup_{t\in \mathbb{R}}
	\left \|
		\int_{-\infty}^t e^{-(t-\tau)(-\Delta)^{\frac{\alpha}{2}}}
		(\mathcal{R}^{\perp}\theta(\tau) \cdot \nabla \theta(\tau)) d\tau
	\right\|_X
	\leqslant
	C\left(\sup_{t\in \mathbb{R}}\|\theta(t)\|_X\right)^2.
\end{equation*}
As another approach, let us consider the successive approximation defined by the transport diffusion type equation
\begin{equation}\label{1.7.1}
	\begin{cases}
		\partial_t \theta^{(n+1)}
			+(-\Delta)^{\frac{\alpha}{2}} \theta^{(n+1)}
			+u^{(n)} \cdot \nabla \theta^{(n+1)}
			=S_{n+4}F,
			&\qquad t\in \mathbb{R},x\in \mathbb{R}^2,\\
			u^{(n)}=\mathcal{R}^{\perp} \theta^{(n)},
			&\qquad t\in \mathbb{R}, x\in\mathbb{R}^2.\\
	\end{cases}
\end{equation}
Then, we can obtain the a priori estimates for the approximation solutions by the energy method. However, it seems to be difficult to construct a time periodic solution $\theta^{(n+1)}$ of (\ref{1.7.1}) when $\theta^{(n)}$ is determined.
Therefore, we are not able to proceed in parallel with the energy method of the initial value problem for the supercritical case.

We now introduce an idea to overcome these difficulties and get a time periodic solution.
{Our idea is to consider the successive approximation for the solution to (\ref{QG}) together with the initial data satisfying a necessary condition which ensures the existence of time periodic solutions.

We explain the necessary condition for the initial data by
using the idea by Geissert-Hieber-Nguyen \cite{GHN}.
In \cite{GHN}, they considered the time periodic problem of the abstract lienear equation}
\begin{equation}\label{1.8}
	\begin{cases}
		\partial_t u+Au=F, \qquad & t>0,\\
		u(0)=u_0, \qquad &t=0,
	\end{cases}
\end{equation}
where $A$ denotes by a closed operator satisfying some conditions and $F$ is given $T$-time periodic external force.
It is proved in \cite{GHN} that there exist a initial data $u_0=u_0(F)$ such that (\ref{1.8}) admits a $T$-time periodic solution
\begin{equation*}\label{1.9}
	u(t)=e^{-tA}u_0+\int_0^t e^{-(t-\tau)A}F(\tau) d\tau
\end{equation*}
in some interpolation spaces
by noting that $u_0$, which ensures the existence of a time periodic solution, should satisfy
\begin{equation*}
	\left( 1-e^{-TA} \right)u_0=\int_0^T e^{-(T-\tau)A}F(\tau) d\tau.
\end{equation*}
Our approach is to incorporate this idea to the successive approximation of (\ref{QG})
and we define $\{\theta_0^{(n)}\}_{n=0}^{\infty}$ and $\{\theta^{(n)}\}_{n=0}^{\infty}$ inductively by
\begin{equation}\label{1.10}
	\begin{cases}
		\partial_t \theta^{(n+1)}
			+(-\Delta)^{\frac{\alpha}{2}} \theta^{(n+1)}
			+u^{(n)} \cdot \nabla \theta^{(n+1)}
			=S_{n+4}F,
			&\qquad 0<t\leqslant T,x\in \mathbb{R}^2,\\
			u^{(n)}=\mathcal{R}^{\perp} \theta^{(n)},
			&\qquad 0\leqslant t\leqslant T, x\in\mathbb{R}^2,\\
		\theta^{(n+1)}(0,x)=S_{n+4}\theta_0^{(n+1)},
			&\qquad x\in \mathbb{R}^2,
	\end{cases}
\end{equation}
where $\theta_0^{(n+1)}$ satisfies
\begin{equation*}\label{1.11}
	\begin{split}
		(1-e^{-T(-\Delta)^{\frac{\alpha}{2}}})\theta_0^{(n+1)}
		&=\int_0^Te^{-(T-\tau)(-\Delta)^{\frac{\alpha}{2}}}(S_{n+3}F(\tau)-u^{(n-1)}(\tau) \cdot \nabla \theta^{(n)}(\tau)) d\tau\\
		&=\theta^{(n)}(T)-e^{-T(-\Delta)^{\frac{\alpha}{2}}}\theta^{(n)}(0).
	\end{split}
\end{equation*}
(See in Section \ref{s3} for the precise definition.)
We then gain the uniform boundedness of the sequences $\{\theta_0^{(n)}\}_{n=0}^{\infty}$ and $\{\theta^{(n)}\}_{n=0}^{\infty}$ by the energy method
if the size of the time periodic external force is sufficiently small.
{Then, we get a continuous in time solution $\theta$ on $[0,T]$  satisfying $\theta(T)=\theta(0)=\theta_0$ by  converging the sequences
and we obtain a $T$-time periodic solution by extending the solution periodically in time.}
We can also prove the uniqueness by the similar argument as in the convergence part.

This paper is organized as follows. In Section \ref{s2}, we summarize some notations
and introduce lemmas which are key ingredients of the proof of the main results.
In Section \ref{s3}, we prove Theorem \ref{t1-1}.

Throughout this paper, we denote by $C$ the constant, which may differ in each line. In particular, $C=C(a_1,...,a_n)$ means that $C$ depends only on $a_1,...,a_n$. We define a commutator for two operators $A$ and $B$ as $[A,B]=AB-BA$.

\section{Preliminaries}\label{s2}
Let $\mathscr{S}(\mathbb{R}^2)$ be the set of all Schwartz functions on $\mathbb{R}^2$ and let $\mathscr{S}'(\mathbb{R}^2)$ be the set of all tempered distributions on $\mathbb{R}^2$.
For $f\in \mathscr{S}(\mathbb{R}^2)$, we define the Fourier transform and the inverse Fourier transform of $f$ by
\begin{equation*}\label{2.0}
\begin{array}{cc}
		\mathscr{F}[f](\xi)=\widehat{f}(\xi):=\displaystyle\int_{\mathbb{R}^2}e^{-i\xi\cdot x}f(x)\ dx, &	\mathscr{F}^{-1}[f](x):=\dfrac{1}{(2\pi)^2}\displaystyle\int_{\mathbb{R}^2}e^{i\xi\cdot x}f(\xi)\ d\xi,
\end{array}
\end{equation*}
respectively.
$\{\varphi_j\}_{j\in \mathbb{Z}}$ is called the homogeneous Littlewood-Paley decomposition if $\varphi_0\in \mathscr{S}(\mathbb{R}^2)$ satisfy ${\rm supp}\ \widehat{\varphi_0} \subset \{2^{-1}\leqslant |\xi|\leqslant 2\}$, $0\leqslant \widehat{\varphi_0}\leqslant 1$ and
\begin{equation*}\label{2.1}
		 \sum_{j\in\mathbb{Z}}\widehat{\varphi_j}(\xi)=1, \quad \quad \xi \in \mathbb{R}^2\setminus\{0\},
\end{equation*}
where $\widehat{\varphi_j}(\xi)=\widehat{\varphi_0}(2^{-j}\xi)$.
Let us write
\begin{equation*}\label{2.2}
		\Delta_jf:= \varphi_j*f
\end{equation*}
for $j\in \mathbb{Z}$ and $f\in \mathscr{S}'(\mathbb{R}^2)$.
Using the homogeneous Littlewood-Paley decomposition, we define the Besov spaces. For $1\leqslant p,q\leqslant \infty$ and $s\in \mathbb{R}$, the homogeneous Besov space $\dot{B}^s_{p,q}(\mathbb{R}^2)$ is defined by
\begin{equation*}\label{2.3}
	\begin{split}
		\dot{B}^s_{p,q}(\mathbb{R}^2)
			&:=\left\{f\in \mathscr{S}_0'(\mathbb{R}^2)\ ;\
		 	\|f\|_{\dot{B}^s_{p,q}}<\infty.\right\},\\
		  \|f\|_{\dot{B}^s_{p,q}}
			&:=\left\|\left\{2^{js}\|\Delta_j f\|_{L^p}\right\}_{j\in \mathbb{Z}} \right\|_{l^q(\mathbb{Z})},
	\end{split}
\end{equation*}
where $\mathscr{S}_0'(\mathbb{R}^2)$ is the dual space of
\begin{equation*}\label{2.4}
	\mathscr{S}_0(\mathbb{R}^2):=\left\{ f\in \mathscr{S}(\mathbb{R}^2)\ ;\ \int_{\mathbb{R}^2}x^{\gamma}f(x) dx=0{\rm\ for\ all\ }\gamma\in (\mathbb{N}\cup\{0\})^2. \right\}.
\end{equation*}
Note that $\dot{B}_{p,q}^s(\mathbb{R}^2)$ is a Banach space with respect to the norm $\|\cdot\|_{\dot{B}_{p,q}^s}$.
It is well known that
if $1\leqslant p,q\leqslant \infty$ and $s<2/p$, then we can identify $\dot{B}_{p,q}^s(\mathbb{R}^2)$ as
\begin{equation*}
	\left\{f\in \mathscr{S}'(\mathbb{R}^2)\ ;\ f=\sum_{j\in \mathbb{Z}}\Delta_jf\ {\rm in}\ \mathscr{S}'(\mathbb{R}^2)\ {\rm and}\ \|f\|_{\dot{B}_{p,q}^s}<\infty.\right\}.
\end{equation*}
(See for the detail in \cite{KY} and \cite{Sawano}.)
For $s>0$ and $1\leqslant p,q\leqslant \infty$, the inhomogeneous Besov space $B_{p,q}^s(\mathbb{R}^2)$ is defined by
\begin{equation*}\label{2.6}
	\begin{split}
		B_{p,q}^s(\mathbb{R}^2)&:=\dot{B}_{p,q}^s(\mathbb{R}^2)\cap L^p(\mathbb{R}^2),\\
		\|f\|_{B_{p,q}^s}&:=\|f\|_{\dot{B}_{p,q}^s}+\|f\|_{L^p}.
	\end{split}
\end{equation*}
In this paper, we also use the space-time Besov spaces defined by
\begin{equation*}\label{2.7}
\begin{split}
	\widetilde{L}^r(0,T;\dot{B}_{p,q}^s(\mathbb{R}^2))
		&:=\left\{F:(0,T)\to \mathscr{S}_0'(\mathbb{R}^2)\ ;\ \|F\|_{\widetilde{L}^r(0,T;\dot{B}_{p,q}^s)}<\infty.\right\},\\
	\|F\|_{\widetilde{L}^r(0,T;\dot{B}_{p,q}^s)}
		&:=\left\|\left\{2^{sj}\|\Delta_jF\|_{L^r(0,T;L^p)}\right\}_{j\in\mathbb{Z}}\right\|_{l^{q}(\mathbb{Z})}
\end{split}
\end{equation*}
for $1\leqslant p,q,r\leqslant \infty,s\in\mathbb{R}$ and $0<T\leqslant \infty$.

Next, we introduce the semigroup generated by the fractional Laplacian $(-\Delta)^{\frac{\alpha}{2}}$.
It is given explicitly by using the Fourier transform:
\begin{equation*}\label{2.7.1}
	e^{-t(-\Delta)^{\frac{\alpha}{2}}}f
	=\mathscr{F}^{-1}\left[e^{-t|\xi|^{\alpha}} \widehat{f}(\xi) \right].
\end{equation*}
Then, this semigroup possesses the following properties:
\begin{lemm}\label{t2-1}
	Let $\alpha>0$ and $1\leqslant p,q \leqslant \infty$.
	Then, the followings hold:
	\begin{itemize}
		\item [(1)] There exists a positive constant $C=C(\alpha)$
		such that
		\begin{equation*}\label{2.7.2}
			\left\| e^{-t(-\Delta)^{\frac{\alpha}{2}}} \Delta_j f \right\|_{L^p}
			\leqslant
			C e^{-C^{-1}2^{\alpha j}t} \|\Delta_jf\|_{L^p}
		\end{equation*}
		holds for all $t>0$, $j\in \mathbb{Z}$ and $f\in \mathscr{S}'_0(\mathbb{R}^2)$ with $\Delta_j f\in L^p(\mathbb{R}^2)$.
		\item [(2)]  Let $s_1,s_2\in \mathbb{R}$ satisfy $s_1\leqslant s_2$.
		Then, there exists a positive constant $C=C(\alpha,s_1,s_2)$
		such that
		\begin{equation*}\label{2.7.3}
			2^{s_2 j}\left\|e^{-t(-\Delta)^{\frac{\alpha}{2}}}\Delta_j f\right\|_{L^p}
			\leqslant
			C t^{-\frac{s_2-s_1}{\alpha}}2^{s_1 j}\|\Delta_j f\|_{L^p}
		\end{equation*}
		holds for all $t>0$, $j\in \mathbb{Z}$ and $f\in \mathscr{S}'_0(\mathbb{R}^2)$ with $\Delta_j f\in L^p(\mathbb{R}^2)$.
		In particular, it holds
		\begin{equation*}\label{2.7.3.1}
			\left\| e^{-t(-\Delta)^{\frac{\alpha}{2}}}f\right\|_{\dot{B}_{p,q}^{s_2}}
			\leqslant
			C
			t^{-\frac{s_2-s_1}{\alpha}}\|f\|_{\dot{B}_{p,q}^{s_1}}
		\end{equation*}
		for all $t>0$ and $f\in \dot{B}_{p,q}^{s_1}(\mathbb{R}^2)$.
		\item [(3)] Let $s\in \mathbb{R}$.
		Then, for each $f\in \dot{B}_{p,q}^s(\mathbb{R}^2)$, it holds
		\begin{equation*}\label{2.7.4}
			\lim_{t\to \infty} \left\|e^{-t(-\Delta)^{\frac{\alpha}{2}}}f\right\|_{\dot{B}_{p,q}^s}=0.
		\end{equation*}
	\end{itemize}
\end{lemm}
\begin{proof}
	(1) is proved in \cite{HK} and \cite{Zhang}.
	(2) is immediately obtained by (1) and
	\begin{equation*}\label{2.7.5}
		2^{s_2j}e^{-C^{-1}2^{\alpha j}t}\leqslant Ct^{-\frac{s_2-s_1}{\alpha}}2^{s_1j}.
	\end{equation*}
	Let us prove (3).
	The density property yields that for any $\varepsilon>0$, there exists a $f_{\varepsilon}\in \mathscr{S}_0(\mathbb{R}^2)$ such that
	$\|f_{\epsilon}-f\|_{\dot{B}_{p,q}^s}<\varepsilon$.
	Then, we see that
	\begin{equation*}
		\begin{split}
			\left\|e^{-t(-\Delta)^{\frac{\alpha}{2}}}f\right\|_{\dot{B}_{p,q}^s}
			&\leqslant
				\left\|e^{-t(-\Delta)^{\frac{\alpha}{2}}}(f-f_{\varepsilon})\right\|_{\dot{B}_{p,q}^s}
				+\left\|e^{-t(-\Delta)^{\frac{\alpha}{2}}}f_{\varepsilon}\right\|_{\dot{B}_{p,q}^s}\\
			&\leqslant
				C\|f-f_{\epsilon}\|_{\dot{B}_{p,q}^s}+Ct^{-\frac{1}{\alpha}}\|f_{\varepsilon}\|_{\dot{B}_{p,q}^{s-1}}\\
			&<C\varepsilon+Ct^{-\frac{1}{\alpha}}\|f_{\varepsilon}\|_{\dot{B}_{p,q}^{s-1}},
		\end{split}
	\end{equation*}
	which implies
	\begin{equation*}
		\limsup_{t\to \infty}\left\|e^{-t(-\Delta)^{\frac{\alpha}{2}}}f\right\|_{\dot{B}_{p,q}^s}
		\leqslant
		C\varepsilon.
	\end{equation*}
	Since $\varepsilon>0$ is arbitrary, the proof is completed.
\end{proof}
Next, we derive some bilinear estimates.
We first recall the definition and basic properties of the Bony paraproduct formula.
For $f,g\in \mathscr{S}_0(\mathbb{R}^2)$, we decompose the product $fg$ as
\begin{equation*}\label{2.8}
	fg=T_fg+R(f,g)+T_gf,
\end{equation*}
where
\begin{equation*}\label{2.9}
	T_fg:=\sum_{l\in \mathbb{Z}}S_lf\Delta_lg,
	\qquad R(f,g):=\sum_{l\in \mathbb{Z}}\sum_{|k-l|\leqslant2}\Delta_kf\Delta_lg.
\end{equation*}
Here, $S_lf$ is defined by
\begin{equation*}\label{2.10}
	S_lf:=\sum_{k\leqslant l-3}\Delta_kf,\qquad l\in \mathbb{Z}.
\end{equation*}
Then, considering the supports of the functions of the Fourier side, we have
\begin{equation*}\label{2.10.1}
	\Delta_jT_{f}g=\sum_{l:|l-j|\leqslant 3}\Delta_j(S_lf\Delta_lg),
	\qquad \Delta_jR(f,g)=\sum_{(k,l):\substack{\max\{j,k\}\geqslant l-3,\\ |k-l|\leqslant2}}\Delta_j(\Delta_kf\Delta_lg).
\end{equation*}
Using them, we have
for $T>0$, $1\leqslant p,q \leqslant \infty$ and $s_1,s_2\in \mathbb{R}$ with $s_1<0$ (if $q=1$, then $s_1\leqslant 0$) that
\begin{equation}\label{2.10.2}
	2^{(s_1+s_2-\frac{2}{p})j}\|\Delta_jT_{f}g\|_{L^{\infty}(0,T;L^p)}
	\leqslant C\|f\|_{\widetilde{L}^{\infty}(0,T\dot{B}_{p,q}^{s_1})}
	\sum_{|l-j|\leqslant 3}2^{s_2l}\|\Delta_lg\|_{L^{\infty}(0,T;L^p)}
\end{equation}
and it also holds for $1\leqslant p,q \leqslant \infty$ and $s_1,s_2\in \mathbb{R}$ with $s_1+s_2>0$
\begin{equation}\label{2.10.3}
	\|R(f,g)\|_{\widetilde{L}^{\infty}(0,T;\dot{B}_{p,q}^{s_1+s_2})}
	\leqslant C\|f\|_{\widetilde{L}^{\infty}(0,T;\dot{B}_{p,q}^{s_1})}
	\|g\|_{\widetilde{L}^{\infty}(0,T;\dot{B}_{p,q}^{s_2})}.
\end{equation}
See \cite{Bahouri} for the idea of the proof of these estimates.
From easy applications of (\ref{2.10.2}) and (\ref{2.10.3}),
we obtain the following lemma:
\begin{lemm}\label{t2-2}
	Let $2\leqslant p\leqslant \infty$ and $1\leqslant q \leqslant \infty$.
	Let $s_1,s_2\in \mathbb{R}$ satisfy $s_1+s_2>0$ and $s_1,s_2<2/p$.
	Then, there exists a positive constant $C=C(p,q,s_1,s_2)$
	such that
	\begin{equation*}\label{2.11}
		\|fg\|_{\widetilde{L}^{\infty}(0,T;\dot{B}_{p,q}^{s_1+s_2-\frac{2}{p}})}
		\leqslant
		C
		\|f\|_{\widetilde{L}^{\infty}(0,T;\dot{B}_{p,q}^{s_1})}
		\|g\|_{\widetilde{L}^{\infty}(0,T;\dot{B}_{p,q}^{s_2})}
	\end{equation*}
	holds for all $T>0$,
	$f\in \widetilde{L}^{\infty}(0,T;\dot{B}_{p,q}^{s_1}(\mathbb{R}^2))$
	and $g\in \widetilde{L}^{\infty}(0,T;\dot{B}_{p,q}^{s_2}(\mathbb{R}^2))$.
\end{lemm}
By the standard argument of the proof of commutator estimates
(see for instance \cite{Bahouri}, \cite{Miura}), we get the following lemma:
\begin{lemm}\label{t2-3}
	Let $2\leqslant p\leqslant \infty$ and $1\leqslant q \leqslant \infty$.
	Let $s_1,s_2\in \mathbb{R}$ satisfy $s_1+s_2>0$, $0<s_1<1+2/p$ and $s_2<2/p$.
	Then, there exists a positive constant $C=C(p,q,s_1,s_2)$
	such that
	\begin{equation*}\label{2.12}
		\left\|\left\{2^{(s_1+s_2-\frac{2}{p})j}\|[f,\Delta_j]g\|_{L^{\infty}(0,T;L^p)}\right\}_{j\in \mathbb{Z}}\right\|_{l^{q}(\mathbb{Z})}
		\leqslant
		C
		\|f\|_{\widetilde{L}^{\infty}(0,T;\dot{B}_{p,q}^{s_1})}
		\|g\|_{\widetilde{L}^{\infty}(0,T;\dot{B}_{p,q}^{s_2})}
	\end{equation*}
	holds for all $T>0$,
	$f\in \widetilde{L}^{\infty}(0,T;\dot{B}_{p,q}^{s_1}(\mathbb{R}^2))$
	and $g\in \widetilde{L}^{\infty}(0,T;\dot{B}_{p,q}^{s_2}(\mathbb{R}^2))$.
\end{lemm}
{The next lemma helps us to control the product term which will appear in equations of the perturbation
such as (\ref{3.7}), (\ref{3.47}) and (\ref{3.56}) below.}
\begin{lemm}\label{t2-4}
	Let $\lambda>0$, $\alpha>0$, $\beta \leqslant \alpha$,
	$2\leqslant p\leqslant \infty$ and $1\leqslant q \leqslant \infty$.
	Let $s_1,s_2\in \mathbb{R}$ satisfy $s_1+s_2>0$, $2/p<s_1<2/p+\alpha$ and $s_2<2/p$.
	Then, there exists a positive constant $C=C(\lambda,\alpha,\beta,p,q,s_1,s_2)$ such that
	\begin{equation}\label{2.13}
		\begin{split}
			&\left\|\left\{
				\int_0^T2^{\beta j} e^{-\lambda 2^{\alpha j}(T-\tau)}2^{(s_1+s_2-\frac{2}{p})j}
				\|\Delta_j(f(\tau) e^{-\tau(-\Delta)^{\frac{\alpha}{2}}}g)\|_{L^p}
				d\tau
			\right\}_{j\in \mathbb{Z}}\right\|_{l^q(\mathbb{Z})}\\
			&
			\leqslant
			CT^{1-\frac{\beta}{\alpha}-\frac{1}{\alpha}(s_1-\frac{2}{p})}
			\left\{
				\|f\|_{L^{\infty}(0,T;L^p)}
				+\left(1+T^{\frac{1}{\alpha}(s_1-\frac{2}{p})}\right)
				\|f\|_{\widetilde{L}^{\infty}(0,T;\dot{B}_{p,q}^{s_1})}
			\right\}
			\|g\|_{\dot{B}_{p,q}^{s_2}}
		\end{split}
	\end{equation}
	holds for all $T>0$,
	$f\in L^{\infty}(0,T;L^p(\mathbb{R}^2))\cap \widetilde{L}^{\infty}(0,T;\dot{B}_{p,q}^{s_1}(\mathbb{R}^2))$
	and $g\in \dot{B}_{p,q}^{s_2}(\mathbb{R}^2)$.

	In particular, if $\beta<\alpha$, then the following estimate holds:
	\begin{equation}\label{2.13.0.0}
		\begin{split}
			&\sum_{j\in \mathbb{Z}}
				\int_0^T2^{\beta j} e^{-\lambda 2^{\alpha j}(T-\tau)}
				\|f(\tau) e^{-\tau(-\Delta)^{\frac{\alpha}{2}}}g\|_{\dot{B}_{p,q}^{s_1+s_2-\frac{2}{p}}}
				d\tau\\
			&\qquad
			\leqslant
			CT^{1-\frac{\beta}{\alpha}-\frac{1}{\alpha}(s_1-\frac{2}{p})}
			\left\{
				\|f\|_{L^{\infty}(0,T;L^p)}
				+\left(1+T^{\frac{1}{\alpha}(s_1-\frac{2}{p})}\right)
				\|f\|_{\widetilde{L}^{\infty}(0,T;\dot{B}_{p,q}^{s_1})}
			\right\}
			\|g\|_{\dot{B}_{p,q}^{s_2}}.
		\end{split}
	\end{equation}
\end{lemm}
\begin{rem}\label{t2-4-0}
	Let $1/2<\alpha<1$ and $2\leqslant p < 4/(2\alpha-1)$.
	Then, it immediately follows from (\ref{2.13}) with $s_1=s_{\rm c}:=1+2/p-\alpha$ and the continuous embedding
 	$\widetilde{L}^{\infty}(0,T;\dot{B}_{p,1}^0(\mathbb{R}^2))\hookrightarrow L^{\infty}(0,T;L^p(\mathbb{R}^2))$ that
	\begin{equation}\label{2.13.0}
		\begin{split}
			&\left\|\left\{
				\int_0^T2^{\beta j} e^{-\lambda 2^{\alpha j}(T-\tau)}2^{(s_{\rm c}+s_2-\frac{2}{p})j}
				\|\Delta_j(f(\tau) e^{-\tau(-\Delta)^{\frac{\alpha}{2}}}g)\|_{L^p}
				d\tau
			\right\}_{j\in \mathbb{Z}}\right\|_{l^q(\mathbb{Z})}\\
			&\qquad
			\leqslant
			CT^{1-\frac{\beta}{\alpha}}
			(T^{-\frac{1-\alpha}{\alpha}}+1)
			\|f\|_{X_T^{p,q}}
			\|g\|_{\dot{B}_{p,q}^{s_2}},
		\end{split}
	\end{equation}
	where $X_T^{p,q}:=\widetilde{L}^{\infty}(0,T;\dot{B}_{p,1}^0(\mathbb{R}^2))
	\cap \widetilde{L}^{\infty}(0,T;\dot{B}_{p,q}^{s_{\rm c}}(\mathbb{R}^2))$.
	If $\beta<\alpha$, then (\ref{2.13.0.0}) yields that
	\begin{equation}\label{2.13.6.1}
		\begin{split}
			&\sum_{j\in \mathbb{Z}}
				\int_0^T2^{\beta j} e^{-\lambda 2^{\alpha j}(T-\tau)}
				\|f(\tau) e^{-\tau(-\Delta)^{\frac{\alpha}{2}}}g\|_{\dot{B}_{p,q}^{s_1+s_2-\frac{2}{p}}}
				d\tau\\
			&\qquad
			\leqslant
			CT^{1-\frac{\beta}{\alpha}}
			(T^{-\frac{1-\alpha}{\alpha}}+1)
			\|f\|_{X_T^{p,q}}
			\|g\|_{\dot{B}_{p,q}^{s_2}}.
		\end{split}
	\end{equation}
\end{rem}
	\begin{proof}[Proof of Lemma \ref{t2-4}]
		First, we prove (\ref{2.13})
		It follows from an inequality of (\ref{2.10.2}) type and (2) of Lemma \ref{t2-1} that
		\begin{equation}\label{2.13.1}
			\begin{split}
				&\int_0^T2^{\beta j} e^{-\lambda 2^{\alpha j}(T-\tau)}2^{(s_1+s_2-\frac{2}{p})j}
					\|\Delta_jT_{f(\tau)} e^{-\tau(-\Delta)^{\frac{\alpha}{2}}}g\|_{L^p}
					d\tau\\
				&\qquad\leqslant
					C\int_0^T2^{\beta j} e^{-\lambda 2^{\alpha j}(T-\tau)}
					\|f(\tau)\|_{\dot{B}_{p,1}^{\frac{2}{p}}}
					\sum_{|l-j|\leqslant 3}2^{((s_1-\frac{2}{p})+s_2)l}\|e^{-\tau(-\Delta)^{\frac{\alpha}{2}}}\Delta_lg\|_{L^p}
					d\tau\\
				&\qquad\leqslant
					C\int_0^T2^{\beta j} e^{-\lambda 2^{\alpha j}(T-\tau)}
					\tau^{-\frac{1}{\alpha}(s_1-\frac{2}{p})}
					d\tau
					\|f\|_{\widetilde{L}^{\infty}(0,T;\dot{B}_{p,1}^{\frac{2}{p}})}
					\sum_{|l-j|\leqslant 3}2^{s_2l}\|\Delta_lg\|_{L^p}.
			\end{split}
		\end{equation}
		By virtue of $s_1-2/p<\alpha$ and $\beta\leqslant \alpha$, it is easy to see that
		\begin{equation*}\label{2.13.2}
			\sup_{j\in \mathbb{Z}}
			\int_0^T2^{\beta j} e^{-\lambda 2^{\alpha j}(T-\tau)}
			\tau^{-\frac{1}{\alpha}(s_1- \frac{2}{p})}
			d\tau
			\leqslant
			CT^{1-\frac{\beta}{\alpha}-\frac{1}{\alpha}(s_1- \frac{2}{p})}.
		\end{equation*}
		Hence, taking $l^q(\mathbb{Z})$-norm of (\ref{2.13.1}) and using
		\begin{equation}\label{2.13.2.0}
			\begin{split}
				\|f\|_{\widetilde{L}^{\infty}(0,T;\dot{B}_{p,1}^{\frac{2}{p}})}
				&\leqslant
					\sum_{j\leqslant 0}2^{\frac{2}{p}j}\|\Delta_jf\|_{L^{\infty}(0,T;L^p)}
					+\left\|\left\{2^{(\frac{2}{p}-s_1)j}\right\}\right\|_{l^{\frac{q}{q-1}}(\mathbb{N})}
					\|f\|_{\widetilde{L}^{\infty}(0,T;\dot{B}_{p,q}^{s_1})}\\
				&\leqslant
					C\left(\|f\|_{L^{\infty}(0,T;L^p)}+\|f\|_{\widetilde{L}^{\infty}(0,T;\dot{B}_{p,q}^{s_1})}\right),
			\end{split}
		\end{equation}
		we obtain that
		\begin{equation}\label{2.13.2.1}
			\begin{split}
				&\left\|\left\{\int_0^T2^{\beta j} e^{-\lambda 2^{\alpha j}(T-\tau)}2^{(s_1+s_2-\frac{2}{p})j}
					\|\Delta_jT_{f(\tau)} e^{-\tau(-\Delta)^{\frac{\alpha}{2}}}g\|_{L^p}
					d\tau\right\}_{j\in \mathbb{Z}}\right\|_{l^q(\mathbb{Z})}\\
				&\quad\leqslant
				CT^{1-\frac{\beta}{\alpha}-\frac{1}{\alpha}(s_1-\frac{2}{p})}
				\left(\|f\|_{L^{\infty}(0,T;L^p)}+\|f\|_{\widetilde{L}^{\infty}(0,T;\dot{B}_{p,q}^{s_1})}\right)
				\left\|g\right\|_{\dot{B}_{p,q}^{s_2}}.
			\end{split}
		\end{equation}
		Since it holds
		\begin{equation}\label{2.13.3}
			\begin{split}
				&\int_0^T2^{\beta j} e^{-\lambda 2^{\alpha j}(T-\tau)}2^{(s_1+s_2-\frac{2}{p})j}
					\left\|\Delta_jR(f(\tau),e^{-\tau(-\Delta)^{\frac{\alpha}{2}}}g)\right\|_{L^p}d\tau\\
				&\leqslant\int_0^T2^{\beta j} e^{-\lambda 2^{\alpha j}(T-\tau)} d\tau
					2^{(s_1+s_2-\frac{2}{p})j}
					\left\|\Delta_jR(f,e^{-t(-\Delta)^{\frac{\alpha}{2}}}g)\right\|_{L^{\infty}_t(0,T;L^p)},
			\end{split}
		\end{equation}
		taking $l^q(\mathbb{Z})$-norm of (\ref{2.13.3}), we see by (\ref{2.10.3}) and (2) of Lemma \ref{t2-1} that
		\begin{equation}\label{2.13.4}
			\begin{split}
				&\left\|\left\{
					\int_0^T2^{\beta j} e^{-\lambda 2^{\alpha j}(T-\tau)}2^{(s_1+s_2-\frac{2}{p})j}
					\left\|\Delta_jR(f(\tau),e^{-\tau(-\Delta)^{\frac{\alpha}{2}}}g)\right\|_{L^p}
					d\tau
				\right\}_{j\in \mathbb{Z}}\right\|_{l^q(\mathbb{Z})}\\
				&\quad\leqslant
					\sup_{j\in \mathbb{Z}}\int_0^T2^{\beta j} e^{-\lambda 2^{\alpha j}(T-\tau)} d\tau
					\left\|R(f,e^{-t(-\Delta)^{\frac{\alpha}{2}}}g)\right\|_{\widetilde{L}^{\infty}_t(0,T;\dot{B}_{p,q}^{s_1+s_2-\frac{2}{p}})}\\
				&\quad\leqslant
					CT^{1-\frac{\beta}{\alpha}}
					\|f\|_{\widetilde{L}^{\infty}(0,T;\dot{B}_{p,q}^{s_1})}
					\left\|g\right\|_{\dot{B}_{p,q}^{s_2}}.
			\end{split}
		\end{equation}
		Here, we have used
		\begin{equation}\label{3.16}
			\sup_{j\in \mathbb{Z}}\int_0^T 2^{\beta j}e^{\lambda 2^{\alpha j}(T-\tau)} d\tau
			\leqslant
			CT^{1-\frac{\beta}{\alpha}},
			\qquad \beta \leqslant \alpha.
		\end{equation}
		Similarly, it follows from (\ref{2.10.2}) and (2) of Lemma \ref{t2-1} that
		\begin{equation}\label{2.13.5}
			\begin{split}
				&\left\|\left\{
					\int_0^T2^{\beta j} e^{-\lambda 2^{\alpha j}(T-\tau)}2^{(s_1+s_2-\frac{2}{p})j}
					\left\|\Delta_jT_{e^{-\tau(-\Delta)^{\frac{\alpha}{2}}}g}f(\tau)\right\|_{L^p}
					d\tau
				\right\}_{j\in \mathbb{Z}}\right\|_{l^q(\mathbb{Z})}\\
				&\quad\leqslant
					\sup_{j\in \mathbb{Z}}\int_0^T2^{\beta j} e^{-\lambda 2^{\alpha j}(T-\tau)} d\tau
					\left\|T_{e^{-t(-\Delta)^{\frac{\alpha}{2}}}g}f(t)\right\|_
					{\widetilde{L}^{\infty}_t(0,T;\dot{B}_{p,q}^{s_1+s_2-\frac{2}{p}})}\\
				&\quad\leqslant
					CT^{1-\frac{\beta}{\alpha}}
					\|f\|_{\widetilde{L}^{\infty}(0,T;\dot{B}_{p,q}^{s_1})}
					\left\|e^{-t(-\Delta)^{\frac{\alpha}{2}}}g\right\|_{\widetilde{L}^{\infty}(0,T;\dot{B}_{p,q}^{s_2})}\\
				&\quad\leqslant
					CT^{1-\frac{\beta}{\alpha}}
					\|f\|_{\widetilde{L}^{\infty}(0,T;\dot{B}_{p,q}^{s_1})}
					\left\|g\right\|_{\dot{B}_{p,q}^{s_2}}.
			\end{split}
		\end{equation}
		Commbining (\ref{2.13.2.1}), (\ref{2.13.4}) and (\ref{2.13.5}),
		we complete the proof of (\ref{2.13.0}).
		Next, we show (\ref{2.13.0.0}).
		By similar inequality to (\ref{2.10.2}) and (2) of Lemma \ref{t2-1}, we see that
	\begin{equation}\label{2.13.7}
		\begin{split}
			&\sum_{j\in \mathbb{Z}}
				\int_0^T2^{\beta j} e^{-\lambda 2^{\alpha j}(T-\tau)}
				\|T_{f(\tau)} e^{-\tau(-\Delta)^{\frac{\alpha}{2}}}g\|_{\dot{B}_{p,q}^{s_1+s_2-\frac{2}{p}}}
				d\tau\\
			&\quad\leqslant
				C\sum_{j\in \mathbb{Z}}
				\int_0^T2^{\beta j} e^{-\lambda 2^{\alpha j}(T-\tau)}\tau^{-\frac{1}{\alpha}(s_1-\frac{2}{p})}d\tau
				\|f\|_{\widetilde{L}^{\infty}(0,T;\dot{B}_{p,1}^{\frac{2}{p}})}
				\|g\|_{\dot{B}_{p,q}^{s_2}}\\
			&\quad\leqslant
				CT^{1-\frac{\beta}{\alpha}-\frac{1}{\alpha}(s_1-\frac{2}{p})}
				\left(\|f\|_{L^{\infty}(0,T;L^p)}+\|f\|_{\widetilde{L}^{\infty}(0,T;\dot{B}_{p,q}^{s_1})}\right)
				\left\|g\right\|_{\dot{B}_{p,q}^{s_2}}.
		\end{split}
	\end{equation}
	We also obtain from (\ref{2.10.3}) and (2) of Lemma \ref{t2-1} that
	\begin{equation}\label{2.13.8}
		\begin{split}
			&\sum_{j\in \mathbb{Z}}\int_0^T2^{\beta j} e^{-\lambda 2^{\alpha j}(T-\tau)}
				\left\|R(f(\tau),e^{-\tau(-\Delta)^{\frac{\alpha}{2}}}g)\right\|_{\dot{B}_{p,q}^{s_1+s_2-\frac{2}{p}}}d\tau\\
			&\leqslant
				\sum_{j\in \mathbb{Z}}\int_0^T2^{\beta j} e^{-\lambda 2^{\alpha j}(T-\tau)} d\tau
				\left\|R(f,e^{-\tau(-\Delta)^{\frac{\alpha}{2}}}g)\right\|_{\widetilde{L}^{\infty}(0,T;\dot{B}_{p,q}^{s_1+s_2-\frac{2}{p}})}\\
			&\leqslant
				CT^{1-\frac{\beta}{\alpha}}
				\|f\|_{\widetilde{L}^{\infty}(0,T;\dot{B}_{p,q}^{s_1})}
				\left\|g\right\|_{\dot{B}_{p,q}^{s_2}}.
		\end{split}
	\end{equation}
	Here, we have used the following inequalities in (\ref{2.13.7}) and (\ref{2.13.8}):
	\begin{equation}\label{2.13.8.1}
		\begin{split}
			&\sum_{j\in \mathbb{Z}}\int_0^T2^{\beta j}e^{-\lambda 2^{\alpha j}(T-\tau)}\tau^{-\frac{\gamma}{\alpha}} d\tau
				\leqslant CT^{1-\frac{\beta}{\alpha}-\frac{\gamma}{\alpha}},
				\qquad \beta<\alpha,\ \gamma<\alpha,\\
			&\sum_{j\in \mathbb{Z}}\int_0^T2^{\beta j}e^{-\lambda 2^{\alpha j}(T-\tau)} d\tau
				\leqslant CT^{1-\frac{\beta}{\alpha}},
				\qquad \beta<\alpha.
		\end{split}
	\end{equation}
	Similarly, we have
	\begin{equation}\label{2.13.9}
		\begin{split}
			&\sum_{j\in \mathbb{Z}}
				\int_0^T2^{\beta j} e^{-\lambda 2^{\alpha j}(T-\tau)}2^{(s_1+s_2-\frac{2}{p})j}
				\left\|T_{e^{-\tau(-\Delta)^{\frac{\alpha}{2}}}g}f(\tau)\right\|_{\dot{B}_{p,q}^{s_1+s_2-\frac{2}{p}}}
				d\tau\\
			&\quad\leqslant
				CT^{1-\frac{\beta}{\alpha}}
				\|f\|_{\widetilde{L}^{\infty}(0,T;\dot{B}_{p,q}^{s_1})}
				\left\|g\right\|_{\dot{B}_{p,q}^{s_2}}.
		\end{split}
	\end{equation}
	Hence, we complete the proof by commbining (\ref{2.13.7}), (\ref{2.13.8}) and (\ref{2.13.9}).
\end{proof}
To derive some estimates for initial data related to a time periodic solution in the proof of the main results,
we introduce the following two lemmas.
\begin{lemm}\label{t2-5}
	Let $\alpha>0$, $1\leqslant p \leqslant \infty$, $1\leqslant q< \infty$ and $s\in \mathbb{R}$.
	Then, for any $f\in \dot{B}_{p,q}^s(\mathbb{R}^2)\cap \dot{B}_{p,q}^{s-\alpha}(\mathbb{R}^2)$,
	the series
	\begin{equation}\label{2.14}
		u=\sum_{k=0}^{\infty} e^{-Tk(-\Delta)^{\frac{\alpha}{2}}}f
	\end{equation}
	converges in $\dot{B}_{p,q}^s(\mathbb{R}^2)$ and $u$ satisfies
	\begin{equation}\label{2.15}
		(1-e^{-T(-\Delta)^{\frac{\alpha}{2}}})u=f \qquad {\rm in}\ \dot{B}_{p,q}^s(\mathbb{R}^2).
	\end{equation}
	Moreover, there exists a positive constant $C=C(\alpha)$ such that
	\begin{equation}\label{2.16}
		\|u\|_{\dot{B}_{p,q}^s}
		\leqslant
		C
		\left(
			T^{-1}\|f\|_{\dot{B}_{p,q}^{s-\alpha}}
			+\|f\|_{\dot{B}_{p,q}^s}
		\right).
	\end{equation}
\end{lemm}
\begin{proof}
	Let $m,n\in \mathbb{N}$ satisfy $m<n$. Then, it follows from (1) of Lemma \ref{t2-1} that
	\begin{equation}\label{2.16.1}
		\begin{split}
			\left\|
				\sum_{k=m}^n e^{-Tk(-\Delta)^{\frac{\alpha}{2}}}f
			\right\|_{\dot{B}_{p,q}^s}^q
			&=\sum_{j\in \mathbb{Z}}
			\left(
				2^{sj}
				\left\|
					\sum_{k=m}^n e^{-Tk(-\Delta)^{\frac{\alpha}{2}}}\Delta_jf
				\right\|_{L^p}
			\right)^q\\
			&\leqslant
			\sum_{j\in \mathbb{Z}}
 			\left(
 				2^{sj}
 				\|\Delta_jf\|_{L^p}
 				\sum_{k=m}^n Ce^{-C^{-1}T2^{\alpha j}k}
 			\right)^q,
		\end{split}
	\end{equation}
	where $C$ is the same constant as in (1) of Lemma \ref{t2-1}.
	Since it holds
	\begin{equation}\label{2.16.2}
		\begin{split}
			&\sum_{j\in \mathbb{Z}}
			\left(
				2^{sj}
				\|\Delta_jf\|_{L^p}
				\frac{C}{1-e^{-C^{-1}2^{\alpha j}T}}
			\right)^q\\
			&\quad=
			\sum_{j:2^{\alpha j}T<1}
			\left(
				2^{sj}
				\|\Delta_jf\|_{L^p}
				\frac{C^{-1}2^{\alpha j}T}{1-e^{-C^{-1}2^{\alpha j}T}}\cdot C^2 2^{-\alpha j}T^{-1}
			\right)^q\\
			&\qquad+
			\sum_{j:2^{\alpha j}T\geqslant 1}
			\left(
				2^{sj}
				\|\Delta_jf\|_{L^p}
				\frac{C}{1-e^{-C^{-1}2^{\alpha j}T}}
			\right)^q
			\\
			&\quad\leqslant
			\left(
				\frac{C}{1-e^{-C^{-1}}}
			\right)^q
			\left(
				T^{-q}\|f\|_{\dot{B}_{p,q}^{s-\alpha}}^q+\|f\|_{\dot{B}_{p,q}^s}^q
			\right)\\
			&\quad\leqslant
			\left(
				\frac{2C}{1-e^{-C^{-1}}}
			\right)^q
			\left(
				T^{-1}\|f\|_{\dot{B}_{p,q}^{s-\alpha}}+\|f\|_{\dot{B}_{p,q}^s}\right)^q
			<\infty,
		\end{split}
	\end{equation}
	 we have
	 \begin{equation*}\label{2.16.3}
	 		\begin{split}
				2^{sj}
 				\|\Delta_jf\|_{L^p}
 				\sum_{k=m}^n Ce^{-C^{-1}T2^{\alpha j}k}
				&\leqslant
				2^{sj}
 				\|\Delta_jf\|_{L^p}
 				\sum_{k=0}^{\infty} Ce^{-C^{-1}T2^{\alpha j}k}\\
				&=
				2^{sj}
 				\|\Delta_jf\|_{L^p}
 				\frac{C}{1-e^{-C^{-1}2^{\alpha j}T}}\in l^q(\mathbb{Z}).
	 		\end{split}
	 \end{equation*}
	 Hence, it follows from  (\ref{2.16.1}) and the dominated convergence theorem that
	 \begin{equation*}\label{2.16.4}
		 \limsup_{n,m\to \infty}\left\|
			 \sum_{k=m}^n e^{-Tk(-\Delta)^{\frac{\alpha}{2}}}f
		 \right\|^q_{\dot{B}_{p,q}^s}
		 \leqslant
		 \sum_{j\in \mathbb{Z}}
		 \left(
			 2^{sj}
			 \|\Delta_jf\|_{L^p}
			 \lim_{m,n\to \infty}
			 \sum_{k=m}^n Ce^{-C^{-1}T2^{\alpha j}k}
		 \right)^q
		 =0.
	 \end{equation*}
	 Thus, the series (\ref{2.14}) converges in $\dot{B}_{p,q}^s(\mathbb{R}^2)$ and
	 we find that $u$ satisfies (\ref{2.16}) by
	 \begin{equation*}\label{2.16.5}
	 		\begin{split}
				\|u\|_{\dot{B}_{p,q}^s}^q
				&=\sum_{j\in \mathbb{Z}}
				\left(
					2^{sj}
					\left\|
						\sum_{k=0}^{\infty} e^{-Tk(-\Delta)^{\frac{\alpha}{2}}}\Delta_jf
					\right\|_{L^p}
				\right)^q\\
				&\leqslant
				\sum_{j\in \mathbb{Z}}
				\left(
					2^{sj}
					\sum_{k=0}^{\infty}
					\left\|
						e^{-Tk(-\Delta)^{\frac{\alpha}{2}}}\Delta_jf
					\right\|_{L^p}
				\right)^q\\
				&\leqslant\sum_{j\in \mathbb{Z}}
				\left(
					2^{sj}
					\|\Delta_jf\|_{L^p}
					\sum_{k=0}^{\infty} Ce^{-C^{-1}T2^{\alpha j}k}
				\right)^q\\
				&\leqslant
				\left(
					\frac{2C}{1-e^{-C^{-1}}}
				\right)^q
				\left(
					T^{-1}\|f\|_{\dot{B}_{p,q}^{s-\alpha}}+\|f\|_{\dot{B}_{p,q}^s}
				\right)^q,
	 		\end{split}
	\end{equation*}
	where we have used (\ref{2.16.2}).
	Finally, we show (\ref{2.15}).
	Let
 	\begin{equation*}\label{2.18.2}
 		u_N:=\sum_{k=0}^{N-1}e^{-Tk(-\Delta)^{\frac{\alpha}{2}}}f, \qquad N\in \mathbb{N}.
 	\end{equation*}
 	Note that $u_N$ converges to $u$ in $\dot{B}_{p,q}^s(\mathbb{R}^2)$ as $N\to \infty$.
 	By a simple calculation, we see that
 	\begin{equation}\label{2.18.3}
 		(1-e^{-T(-\Delta)^{\frac{\alpha}{2}}})u_N=f-e^{-TN(-\Delta)^{\frac{\alpha}{2}}}f.
 	\end{equation}
 	Here, it follows from (2) of Lemma \ref{t2-1} that
 	\begin{equation}\label{2.18.4}
 		\begin{split}
 			&\left\|(1-e^{-T(-\Delta)^{\frac{\alpha}{2}}})u_N
 				-(1-e^{-T(-\Delta)^{\frac{\alpha}{2}}})u\right\|_{\dot{B}_{p,q}^s}\\
 			&\qquad \leqslant \|u-u_N\|_{\dot{B}_{p,q}^s}
 				+\left\|e^{-T(-\Delta)^{\frac{\alpha}{2}}}(u-u_N)\right\|_{\dot{B}_{p,q}^s}\\
 			&\qquad \leqslant C\|u-u_N\|_{\dot{B}_{p,q}^s}\\
 			&\qquad \to 0
 		\end{split}
 	\end{equation}
 	as $N\to \infty$ and it holds by (3) of Lemma \ref{t2-1} that
 	\begin{equation}\label{2.18.5}
 		\left\|e^{-TN(-\Delta)^{\frac{\alpha}{2}}}f\right\|_{\dot{B}_{p,q}^s}\to 0
 	\end{equation}
 	as $N\to \infty$.
 	Hence, letting $N\to \infty$ in (\ref{2.18.3}) by (\ref{2.18.4}) and (\ref{2.18.5}), we find that $u$ satisfies (\ref{2.15}).
 	This completes the proof.
\end{proof}

\begin{lemm}\label{t2-6}
	Let $T>0$, $\alpha>0$, $1\leqslant p\leqslant \infty$, $1\leqslant q <\infty$ and $s<2/p$.
	Then, for any $T$-time periodic function $F$ satisfying
	\begin{equation*}
		f(t):=\int_0^t e^{-(t-\tau)(-\Delta)^{\frac{\alpha}{2}}}F(\tau) d\tau
		\in BC((0,\infty); \dot{B}_{p,q}^{s-\alpha}(\mathbb{R}^2)\cap\dot{B}_{p,q}^s(\mathbb{R}^2)),
	\end{equation*}
	there exists a unique element $u_0\in \dot{B}_{p,q}^s(\mathbb{R}^2)$ such that the function
	\begin{equation*}\label{2.19}
		u(t)=e^{-t(-\Delta)^{\frac{\alpha}{2}}}u_0+f(t),\qquad t\geqslant0
	\end{equation*}
	is $T$-time periodic.
	Moreover, $u_0$ satisfies
	\begin{equation}\label{2.20}
		\|u_0\|_{\dot{B}_{p,q}^s}
		\leqslant
		C\left(
		T^{-1}\|f(T)\|_{\dot{B}_{p,q}^{s-\alpha}}
		+\|f(T)\|_{\dot{B}_{p,q}^{s}}
		\right),
	\end{equation}
	where $C$ is the same constant as in Lemma \ref{t2-5}.
\end{lemm}
\begin{proof}
	By Lemma \ref{t2-5}, the series
	\begin{equation*}\label{2.23}
		u_0:=\sum_{k=0}^{\infty} e^{-Tk(-\Delta)^{\frac{\alpha}{2}}}f(T)
	\end{equation*}
	converges in $\dot{B}_{p,q}^s(\mathbb{R}^2)$ and $u_0$ satisfies (\ref{2.20}) and
	\begin{equation}\label{2.23.1}
			(1-e^{-T(-\Delta)^{\frac{\alpha}{2}}})u_0=f(T).
	\end{equation}
	By the periodicity of $F$, we have
	\begin{equation}\label{2.24}
		f(t+T)=f(t)+e^{-t(-\Delta)^{\frac{\alpha}{2}}}f(T), \qquad t>0.
	\end{equation}
	Therefore, it follows from (\ref{2.23.1}) and ({\ref{2.24}}) that
	\begin{equation*}\label{2.25}
		\begin{split}
			u(t+T)
			&=e^{-t(-\Delta)^{\frac{\alpha}{2}}}e^{-T(-\Delta)^{\frac{\alpha}{2}}}u_0+f(t+T)\\
			&=e^{-t(-\Delta)^{\frac{\alpha}{2}}}(u_0-f(T))+f(t)+e^{-t(-\Delta)^{\frac{\alpha}{2}}}f(T)\\
			&=u(t)
		\end{split}
	\end{equation*}
	for all $t>0$.
	Hence, $u(t)$ is $T$-time periodic.
	Next, we prove the uniqueness.
	Let $v_0$ be an arbitrary element of $\dot{B}_{p,q}^s(\mathbb{R}^2)$ such that $v(t):=e^{-t(-\Delta)^{\frac{\alpha}{2}}}v_0+f(t)$ is a $T$-time periodic function.
	Then, since $u_0-v_0=u(NT)-v(NT)=e^{-NT(-\Delta)^{\frac{\alpha}{2}}}(u_0-v_0)$ holds for all $N\in \mathbb{N}$ by the periodicity, we obtain by (3) of Lemma \ref{t2-1} that
	\begin{equation*}\label{2.26}
		\|u_0-v_0\|_{\dot{B}_{p,q}^{s}}
		=\left\|e^{-NT(-\Delta)^{\frac{\alpha}{2}}}(u_0-v_0)\right\|_{\dot{B}_{p,q}^s}
		\to 0
	\end{equation*}
	as $N\to \infty$. Therefore, we have $u_0=v_0$ in $\dot{B}_{p,q}^s(\mathbb{R}^2)$ and this completes the proof.
\end{proof}

Finally, we recall a positivity lemma for the $L^p$-energy of the fractional dissipation:
\begin{lemm}[\cite{CMZ2},\cite{Wu}]\label{t2-7}
	Let $0<\alpha\leqslant 2$ and $2\leqslant p <\infty$.
	Then, there exists a positive constant $\lambda=\lambda(\alpha,p)$
	such that
	\begin{equation*}\label{2.27}
		\int_{\mathbb{R}^2} |\Delta_j f(x)|^{p-2}\Delta_jf(x)(-\Delta)^{\frac{\alpha}{2}}\Delta_jf(x) dx
		\geqslant
		\lambda 2^{\alpha j}\|\Delta_jf\|_{L^p}^p
	\end{equation*}
	for all $j\in \mathbb{Z}$ and $f\in \mathscr{S}'_0(\mathbb{R}^2)$ with
	$\Delta_jf \in L^p(\mathbb{R}^2)$.
\end{lemm}

\section{Proof of Main Results}\label{s3}
In this section, we prove Theorem \ref{t1-1}.
Let $T$, $\alpha$, $p$, $q$ and $r$ satisfy the assumptions of Theorem \ref{t1-1}
and let $\sigma$ satisfy $\alpha-2/p<\sigma<2/p$.
We use the following notation for simplicity in this section:
\begin{equation*}\label{3.0}
	\begin{split}
		s_{\rm c}
			&:=1+\frac{2}{p}-\alpha,\\
		X_T^{p,q}
			&:=\widetilde{L}^{\infty}(0,T;\dot{B}_{p,1}^0(\mathbb{R}^2))
			\cap \widetilde{L}^{\infty}(0,T;\dot{B}_{p,q}^{s_{\rm c}}(\mathbb{R}^2)).
	\end{split}
\end{equation*}
We consider the successive approximation sequences
$\{\theta_0^{(n)}\}_{n=0}^{\infty}\subset \dot{B}_{p,1}^0(\mathbb{R}^2)\cap \dot{B}_{p,q}^{s_{\rm c}}(\mathbb{R}^2)$ of the initial data
and $\{\theta^{(n)}\}_{n=0}^{\infty}\subset X_T^{p,q}$
of solutions to (\ref{QG}) defined inductively as follows:

First, let $\theta_0^{(0)}(x)=0$ and $\theta^{(0)}(t,x)=0$.
Next, if $\theta_0^{(n)}$ and $\theta^{(n)}$ are determined, then we define $\theta_0^{(n+1)}$ and $\theta^{(n+1)}$ by the following linear equation:
\begin{equation}\label{3.1}
	\begin{cases}
		\partial_t \theta^{(n+1)}
			+(-\Delta)^{\frac{\alpha}{2}} \theta^{(n+1)}
			+u^{(n)} \cdot \nabla \theta^{(n+1)}
			=S_{n+4}F,
			&\qquad 0<t\leqslant T,x\in \mathbb{R}^2,\\
		u^{(n)}=\mathcal{R}^{\perp} \theta^{(n)},
			&\qquad 0\leqslant t\leqslant T, x\in\mathbb{R}^2,\\
		\theta^{(n+1)}(0,x)=S_{n+4}\theta_0^{(n+1)},
			&\qquad x\in \mathbb{R}^2,
	\end{cases}
\end{equation}
where $\theta_0^{(n+1)}$ is given by
\begin{equation}\label{3.2}
	\theta_0^{(n+1)}
	:=\sum_{k=0}^{\infty} e^{-Tk(-\Delta)^{\frac{\alpha}{2}}}
	\left( \theta^{(n)}(T)-e^{-T(-\Delta)^{\frac{\alpha}{2}}}\theta^{(n)}(0) \right) .
\end{equation}
For $n\in \mathbb{N}\cup\{0\}$, we put $\psi^{(n)}(t):=\theta^{(n)}(t)-e^{-t(-\Delta)^{\frac{\alpha}{2}}}\theta^{(n)}(0)$ and
\begin{equation*}\label{3.2.1}
	\begin{split}
		A_n
			&:=\max \left\{ \|\theta_0^{(n)}\|_{\dot{B}_{p,1}^0\cap\dot{B}_{p,q}^{s_{\rm c}}},
			\|\theta^{(n)}\|_{X_T^{p,q}} \right\} ,\\
		B_n
			&:=\|\theta_0^{(n+1)}-\theta_0^{(n)}\|_{\dot{B}_{p,q}^{\sigma}}
			+\|\theta^{(n+1)}-\theta^{(n)}\|_{\widetilde{L}^{\infty}(0,T;\dot{B}_{p,q}^{\sigma})}.
	\end{split}
\end{equation*}
The well-definedness of the sequences is assured if the series in (\ref{3.2}) converges in $\dot{B}_{p,1}^0(\mathbb{R}^2)\cap \dot{B}_{p,q}^{s_{\rm c}}(\mathbb{R}^2)$.
In the following lemma, we check the convergence and derive some properties of the sequences.
\begin{lemm}\label{t3-1}
	Let $n$ be an positive integer.
	Assume that $\theta_0^{(n)}\in \dot{B}_{p,1}^0(\mathbb{R}^2)\cap \dot{B}_{p,q}^{s_{\rm c}}(\mathbb{R}^2)$
	and $\theta^{(n)}\in X_T^{p,q}$.
	Then, for every $F\in BC((0,\infty);\dot{B}_{r,\infty}^0(\mathbb{R}^2))$,
	the series in (\ref{3.2}) converges in
	$\dot{B}_{p,1}^0(\mathbb{R}^2)\cap \dot{B}_{p,q}^{s_{\rm c}}(\mathbb{R}^2)$
	and it holds
	\begin{equation}\label{3.3.1}
		(1-e^{-T(-\Delta)^{\frac{\alpha}{2}}})\theta_0^{(n+1)}
		=\theta^{(n)}(T)-e^{-T(-\Delta)^{\frac{\alpha}{2}}}\theta^{(n)}(0)
	\end{equation}
	in $\dot{B}_{p,1}^0(\mathbb{R}^2)\cap \dot{B}_{p,q}^{s_{\rm c}}(\mathbb{R}^2)$.
	Moreover,
	there exist positive constants $\delta_1=\delta_1(\alpha,p,q,r,T)$ and $C_1=C_1(\alpha,p,q,r,T)$ such that
	if $F$ satisfies
	\begin{equation*}
	 	\sup_{t>0}\|F(t)\|_{\dot{B}_{r,\infty}^0}\leqslant \delta_1,
	\end{equation*}
	then it holds
	\begin{equation}\label{3.5}
		\begin{split}
			\sup_{m\in \mathbb{N}\cup\{0\}}\|\theta_0^{(m)}\|_{\dot{B}_{p,1}^0\cap\dot{B}_{p,q}^{s_{\rm c}}}
				&\leqslant 2C_1\sup_{t>0}\|F(t)\|_{\dot{B}_{r,\infty}^0},\\
			\sup_{m\in \mathbb{N}\cup\{0\}}\|\theta^{(m)}\|_{X_T^{p,q}}
				&\leqslant 2C_1\sup_{t>0}\|F(t)\|_{\dot{B}_{r,\infty}^0}.
		\end{split}
	\end{equation}
\end{lemm}
\begin{proof}
	To prove the convergence of the series in (\ref{3.2}) in
	$\dot{B}_{p,1}^0(\mathbb{R}^2)\cap \dot{B}_{p,q}^{s_{\rm c}}(\mathbb{R}^2)$,
	Lemma \ref{t2-5} yields that it suffices to check
	\begin{equation*}\label{3.6}
		\begin{split}
			&\theta^{(n)}(T)-e^{-T(-\Delta)^\frac{\alpha}{2}}\theta^{(n)}(0)=\psi^{(n)}(T)\\
			&\quad\in \left( \dot{B}_{p,q}^{s_{\rm c}}(\mathbb{R}^2)\cap \dot{B}_{p,q}^{s_{\rm c}-\alpha}(\mathbb{R}^2) \right)
			\cap \left( \dot{B}_{p,1}^{0}(\mathbb{R}^2)\cap \dot{B}_{p,1}^{-\alpha}(\mathbb{R}^2) \right).
		\end{split}
	\end{equation*}
	Since $\psi^{(n)}$ satisfies
	\begin{equation}\label{3.7}
		\partial_t\psi^{(n)}+(-\Delta)^{\frac{\alpha}{2}}\psi^{(n)}+u^{(n-1)}\cdot \nabla \psi^{(n)}+u^{(n-1)}\cdot \nabla e^{-t(-\Delta)^{\frac{\alpha}{2}}}\theta^{(n)}(0)=S_{n+4}F,
	\end{equation}
	applying $\Delta_j$ to (\ref{3.7}), we see that
	\begin{equation}\label{3.8}
		\begin{split}
			&\partial_t\Delta_j\psi^{(n)}
			+(-\Delta)^{\frac{\alpha}{2}}\Delta_j\psi^{(n)}\\
			&\qquad=S_{n+4}\Delta_jF
			+[u^{(n-1)},\Delta_j]\cdot \nabla \psi^{(n)}\\
			&\qquad\quad-u^{(n-1)}\cdot \nabla \Delta_j\psi^{(n)}
			-\Delta_j(u^{(n-1)}\cdot \nabla e^{-t(-\Delta)^{\frac{\alpha}{2}}}\theta^{(n)}(0)).
		\end{split}
	\end{equation}
	Multiplying (\ref{3.8}) by $p|\Delta_j \psi^{(n)}|^{p-2}\Delta_j\psi^{(n)}$ and integrating over $\mathbb{R}^2$,
	we have by the H\"older inequality that
	\begin{equation}\label{3.9}
		\begin{split}
			&\frac{d}{dt}(\|\Delta_j\psi^{(n)}(t)\|^p_{L^p})
			+p\int_{\mathbb{R}^2}|\Delta_j \psi^{(n)}(t,x)|^{p-2}\Delta_j \psi^{(n)}(t,x)(-\Delta)^{\frac{\alpha}{2}}\Delta_j\psi^{(n)}(t,x) dx\\
			&\qquad\leqslant
		  Cp\|\Delta_j F(t)\|_{L^p}\|\Delta_j \psi^{(n)}(t)\|_{L^p}^{p-1}
			+p\|[u^{(n-1)}(t),\Delta_j]\cdot \nabla \psi^{(n)}(t)\|_{L^p}\|\Delta_j\psi(t)\|_{L^p}^{p-1}\\
			&\qquad \quad
			-p\int_{\mathbb{R}^2} |\Delta_j \psi^{(n)}(t,x)|^{p-2}\Delta_j \psi^{(n)}(t,x) u^{(n-1)}(t,x)\cdot \nabla \Delta_j \psi^{(n)}(t,x) dx\\
			&\qquad \quad
			+p\|\Delta_j(u^{(n-1)}(t)\cdot \nabla e^{-t(-\Delta)^{\frac{\alpha}{2}}}\theta^{(n)}(0))\|_{L^p}\|\Delta_j\psi^{(n)}(t)\|_{L^p}^{p-1}.
		\end{split}
	\end{equation}
	By Lemma \ref{t2-7}, we obtain that
	\begin{equation}\label{3.10}
		\int_{\mathbb{R}^2}|\Delta_j \psi^{(n)}(t,x)|^{p-2}\Delta_j \psi^{(n)}(t,x)(-\Delta)^{\frac{\alpha}{2}}\Delta_j\psi^{(n)}(t,x) dx
		\geqslant
		\lambda 2^{\alpha j}\|\Delta_j \psi^{(n)}(t)\|_{L^p}^p
	\end{equation}
	for some $\lambda=\lambda(\alpha,p)>0$.
	On the other hand, it follows from the divergence free condition $\nabla \cdot u^{(n-1)}=0$ that
	\begin{equation}\label{3.11}
		\int_{\mathbb{R}^2} |\Delta_j \psi^{(n)}(t,x)|^{p-2}\Delta_j \psi^{(n)}(t,x) u^{(n-1)}(t,x)\cdot \nabla \Delta_j \psi^{(n)}(t,x) dx=0.
	\end{equation}
	Substituiting (\ref{3.10}) and (\ref{3.11}) into (\ref{3.9}), we have
	\begin{equation*}\label{3.13}
		\begin{split}
			&\frac{d}{dt}\|\Delta_j\psi^{(n)}(t)\|_{L^p}
			+\lambda 2^{\alpha j}\|\Delta_j \psi^{(n)}(t)\|_{L^p}\\
			&\qquad\leqslant
			C\|\Delta_j F(t)\|_{L^p}
			+\|[u^{(n-1)}(t),\Delta_j]\cdot \nabla \psi^{(n)}(t)\|_{L^p}\\
			&\quad \qquad+\|\Delta_j(u^{(n-1)}(t)\cdot \nabla e^{-t(-\Delta)^{\frac{\alpha}{2}}}\theta^{(n)}(0))\|_{L^p},
		\end{split}
	\end{equation*}
	which implies that
	\begin{equation}\label{3.14}
		\begin{split}
			\|\Delta_j\psi^{(n)}(T)\|_{L^p}
			&\leqslant C\int_0^T e^{-\lambda2^{\alpha j}(T-\tau)}\|\Delta_j F(\tau)\|_{L^p} d\tau\\
			&\quad+\int_0^T e^{-\lambda2^{\alpha j}(T-\tau)}\|[u^{(n-1)}(\tau),\Delta_j]\cdot \nabla \psi^{(n)}(\tau)\|_{L^p} d\tau\\
			&\quad+\int_0^T e^{-\lambda2^{\alpha j}(T-\tau)}\|\Delta_j(u^{(n-1)}(\tau)\cdot \nabla e^{-\tau(-\Delta)^{\frac{\alpha}{2}}}\theta^{(n)}(0))\|_{L^p} d\tau.
		\end{split}
	\end{equation}
	Let $s\in \{s_{\rm c},s_{\rm c}-\alpha\}$.
	Multiplying (\ref{3.14}) by $2^{sj}$, we have
	\begin{equation}\label{3.15}
		\begin{split}
			&2^{sj}\|\Delta_j\psi^{(n)}(T)\|_{L^p}\\
			&\qquad
				\leqslant C\int_0^T 2^{(s+\frac{2}{r}-\frac{2}{p})j}e^{-\lambda2^{\alpha j}(T-\tau)}\|\Delta_j F(\tau)\|_{L^r} d\tau\\
			&\qquad\quad
				+\int_0^T 2^{(\alpha+s-s_{\rm c})j}e^{-\lambda2^{\alpha j}(T-\tau)} d\tau\\
			&\qquad \qquad \qquad
				\times 2^{(2s_{\rm c}-1-\frac{2}{p})j}\|[u^{(n-1)},\Delta_j]\cdot \nabla \psi^{(n)}\|_{L^{\infty}(0,T;L^p)}\\
			&\qquad \quad
				+\int_0^T 2^{(\alpha+s-s_{\rm c})j}e^{-\lambda2^{\alpha j}(T-\tau)}\\
			&\qquad \qquad \qquad
				\times 2^{(2s_{\rm c}-1-\frac{2}{p})j}
				\|\Delta_j(u^{(n-1)}(\tau)\cdot \nabla e^{-\tau(-\Delta)^{\frac{\alpha}{2}}}\theta^{(n)}(0))\|_{L^p} d\tau.\\
			\end{split}
	\end{equation}
	Taking $l^{q}(\mathbb{Z})$-norm of (\ref{3.15}), we see by (\ref{3.16}) that
	\begin{equation*}\label{3.17}
		\begin{split}
			&\|\psi^{(n)}(T)\|_{\dot{B}_{p,q}^s}\\
			&\leqslant C\sum_{j\in \mathbb{Z}}\int_0^T 2^{(s+\frac{2}{r}-\frac{2}{p})j}e^{-\lambda2^{\alpha j}(T-\tau)}d\tau
				\sup_{t>0}\|F(t)\|_{\dot{B}_{r,\infty}^0}\\
			&\quad
				+CT^{\frac{s_{\rm c}-s}{\alpha}}
				\left\|\left\{
					2^{(2s_{\rm c}-1-\frac{2}{p})j}\|[u^{(n-1)},\Delta_j]\cdot \nabla \psi^{(n)}\|_{L^{\infty}(0,T;L^p)}
				\right\}_{j\in \mathbb{Z}}\right\|_{l^q(\mathbb{Z})}\\
			&\quad
				+\left\|\left\{\int_0^T
				2^{(\alpha+s-s_{\rm c})j}e^{-\lambda2^{\alpha j}(T-\tau)}\right.\right.\\
			&\qquad \quad
				\left.\left.
					\times 2^{(2s_{\rm c}-1-\frac{2}{p})j}
					\|\Delta_j(u^{(n-1)}(\tau)\cdot \nabla e^{-\tau(-\Delta)^{\frac{\alpha}{2}}}\theta^{(n)}(0))\|_{L^p} d\tau
				\right\}_{j\in \mathbb{Z}}\right\|_{l^q(\mathbb{Z})}.
		\end{split}
	\end{equation*}
	Hence, it follows from Lemma \ref{t2-3} and (\ref{2.13.0}) that
	\begin{equation*}\label{3.19}
		\begin{split}
			\|\psi^{(n)}(T)\|_{\dot{B}_{p,q}^s}
			&\leqslant CT^{1-\frac{1}{\alpha}(s+\frac{2}{r}-\frac{2}{p})}
				\sup_{t>0}\|F(t)\|_{\dot{B}_{r,\infty}^0}\\
			&\quad+CT^{\frac{s_{\rm c}-s}{\alpha}}
				\|u^{(n-1)}\|_{\widetilde{L}^{\infty}(0,T;\dot{B}_{p,q}^{s_{\rm c}})}
				\|\psi^{(n)}\|_{\widetilde{L}^{\infty}(0,T;\dot{B}_{p,q}^{s_{\rm c}})}\\
			&\quad+CT^{\frac{s_{\rm c}-s}{\alpha}}
				(T^{-\frac{1-\alpha}{\alpha}}+1)
				\|u^{(n-1)}\|_{X_T^{p,q}}
				\|\theta^{(n)}(0)\|_{\dot{B}_{p,q}^{s_{\rm c}}}.
		\end{split}
	\end{equation*}
	Using
	\begin{equation*}\label{3.20}
		\begin{split}
			\|\psi^{(n)}\|_{\widetilde{L}^{\infty}(0,T;\dot{B}_{p,q}^{s_{\rm c}})}
			&\leqslant
				\|\theta^{(n)}\|_{\widetilde{L}^{\infty}(0,T;\dot{B}_{p,q}^{s_{\rm c}})}
				+\|e^{-t(-\Delta)^{\frac{\alpha}{2}}}\theta^{(n)}(0)\|_{\widetilde{L}^{\infty}(0,T;\dot{B}_{p,q}^{s_{\rm c}})}\\
			&\leqslant
				\|\theta^{(n)}\|_{\widetilde{L}^{\infty}(0,T;\dot{B}_{p,q}^{s_{\rm c}})}
				+C\|\theta^{(n)}(0)\|_{\dot{B}_{p,q}^{s_{\rm c}}}\\
			&\leqslant
				\|\theta^{(n)}\|_{X_T^{p,q}}
				+C\|\theta_0^{(n)}\|_{\dot{B}_{p,q}^{s_{\rm c}}}\\
			&\leqslant
				CA_n
		\end{split}
	\end{equation*}
	and the boundedness of the Riesz transform on the homogeneous  space-time Besov spaces, we obtain
	\begin{equation}\label{3.21}
		\begin{split}
			\|\psi^{(n)}(T)\|_{\dot{B}_{p,q}^s}
			&\leqslant CT^{1-\frac{1}{\alpha}(s+\frac{2}{r}-\frac{2}{p})}
				\sup_{t>0}\|F(t)\|_{\dot{B}_{r,\infty}^0}\\
			&\quad+CT^{\frac{s_{\rm c}-s}{\alpha}}
				\|\theta^{(n-1)}\|_{\widetilde{L}^{\infty}(0,T;\dot{B}_{p,q}^{s_{\rm c}})}
				A_n\\
			&\quad+CT^{\frac{s_{\rm c}-s}{\alpha}}
				(T^{-\frac{1-\alpha}{\alpha}}+1)
				\|\theta^{(n-1)}\|_{X_T^{p,q}}
				\|\theta_0^{(n)}\|_{\dot{B}_{p,q}^{s_{\rm c}}}.
		\end{split}
	\end{equation}
	Hence, it follows from (\ref{3.21}) and Lemma \ref{t2-5} that
	\begin{equation}\label{3.21.1}
		\begin{split}
			\|\theta_0^{(n+1)}\|_{\dot{B}_{p,q}^{s_{\rm c}}}
			&\leqslant
			C(T^{-1}\|\psi^{(n)}(T)\|_{\dot{B}_{p,q}^{s_{\rm c}-\alpha}}
			+\|\psi^{(n)}(T)\|_{\dot{B}_{p,q}^{s_{\rm c}}})\\
			&\leqslant CT^{\frac{1}{\alpha}(2\alpha-1-\frac{2}{r})}\sup_{t>0}\|F(t)\|_{\dot{B}_{r,\infty}^0}
			+C(1+T^{-\frac{1-\alpha}{\alpha}})A_{n-1}A_n.
		\end{split}
	\end{equation}
	Let $s':=s-s_{\rm c}\in \{-\alpha,0\}$.
	Multiplying (\ref{3.14}) by $2^{s'j}$ and
	taking $l^1(\mathbb{Z})$-norm, we see that
	\begin{equation*}\label{3.23}
		\begin{split}
			&\|\psi^{(n)}(T)\|_{\dot{B}_{p,1}^{s'}}\\
			&\qquad
				\leqslant C\sum_{j\in \mathbb{Z}}\int_0^T 2^{(s'+\frac{2}{r}-\frac{2}{p})j}e^{-\lambda2^{\alpha j}(T-\tau)}d\tau\sup_{t>0}\|F(t)\|_{\dot{B}_{r,\infty}^0}\\
			&\qquad\quad
				+\sum_{j\in \mathbb{Z}}\int_0^T 2^{(\alpha+s-2s_{\rm c})j}e^{-\lambda2^{\alpha j}(T-\tau)} d\tau\\
			&\qquad \qquad \qquad
				\times
				\left\|\left\{
					2^{(2s_{\rm c}-1-\frac{2}{p})j}\|[u^{(n-1)},\Delta_j]\cdot \nabla \psi^{(n)}\|_{L^{\infty}(0,T;L^p)}
				\right\}_{j\in \mathbb{Z}}\right\|_{l^q(\mathbb{Z})}\\
			&\qquad \quad
				+\sum_{j\in \mathbb{Z}}\int_0^T 2^{(\alpha+s-2s_{\rm c})j}e^{-\lambda2^{\alpha j}(T-\tau)}\\
			&\qquad \qquad \qquad
				\times
					\|u^{(n-1)}(\tau)\cdot \nabla e^{-\tau(-\Delta)^{\frac{\alpha}{2}}}\theta^{(n)}(0)\|_
					{\dot{B}_{p,q}^{2s_{\rm c}-1-\frac{2}{p}}} d\tau.
		\end{split}
	\end{equation*}
	It follows from Lemma \ref{t2-3} and (\ref{2.13.6.1}) that
	\begin{equation*}\label{3.24}
		\begin{split}
			\|\psi^{(n)}(T)\|_{\dot{B}_{p,1}^{s'}}
			&\leqslant CT^{1-\frac{1}{\alpha}(s'+\frac{2}{r}-\frac{2}{p})}
				\sup_{t>0}\|F(t)\|_{\dot{B}_{r,\infty}^0}\\
			&\quad+CT^{\frac{2s_{\rm c}-s}{\alpha}}
				\|u^{(n-1)}\|_{\widetilde{L}^{\infty}(0,T;\dot{B}_{p,q}^{s_{\rm c}})}
				\|\psi^{(n)}\|_{\widetilde{L}^{\infty}(0,T;\dot{B}_{p,q}^{s_{\rm c}})}\\
			&\quad+CT^{\frac{2s_{\rm c}-s}{\alpha}}
				(T^{-\frac{1-\alpha}{\alpha}}+1)
				\|u^{(n-1)}\|_{X_T^{p,q}}
				\|\theta^{(n)}(0)\|_{\dot{B}_{p,q}^{s_{\rm c}}}.
		\end{split}
	\end{equation*}
	Therefore, by the same argument as above, we have
	\begin{equation*}\label{3.25}
		\begin{split}
			\|\psi^{(n)}(T)\|_{\dot{B}_{p,q}^{s'}}
			&\leqslant
				CT^{\frac{s_{\rm c}}{\alpha}}T^{\frac{s_{\rm c}-s}{\alpha}}T^{\frac{1}{\alpha}(2\alpha-1-\frac{2}{r})}
				\sup_{t>0}\|F(t)\|_{\dot{B}_{r,\infty}^0}\\
			&\quad+CT^{\frac{s_{\rm c}}{\alpha}}T^{\frac{s_{\rm c}-s}{\alpha}}(1+T^{-\frac{1-\alpha}{\alpha}})A_{n-1}A_n
		\end{split}
	\end{equation*}
	and we see that the series in (\ref{3.2}) converges in $\dot{B}_{p,1}^0(\mathbb{R}^2)$ and
	\begin{equation}\label{3.26}
		\begin{split}
			\|\theta_0^{(n+1)}\|_{\dot{B}_{p,1}^{0}}
			&\leqslant
				C(T^{-1}\|\psi^{(n)}(T)\|_{\dot{B}_{p,1}^{-\alpha}}
				+\|\psi^{(n)}(T)\|_{\dot{B}_{p,1}^{0}})\\
			&\leqslant
				CT^{\frac{s_{\rm c}}{\alpha}}T^{\frac{1}{\alpha}(2\alpha-1-\frac{2}{r})}\sup_{t>0}\|F(t)\|_{\dot{B}_{r,\infty}^0}
				+CT^{\frac{s_{\rm c}}{\alpha}}(1+T^{-\frac{1-\alpha}{\alpha}})
				A_{n-1}A_n.
		\end{split}
	\end{equation}
	Combining (\ref{3.21}) and (\ref{3.26}), we find that Lemma \ref{t2-5} implies
	the series in (\ref{3.2}) converges in $\dot{B}_{p,1}^0(\mathbb{R}^2)\cap \dot{B}_{p,q}^{s_{\rm c}}(\mathbb{R}^2)$
	and (\ref{3.3.1}) also holds.

	Next, we prove (\ref{3.5}).
	Since $\Delta_j \theta^{(n+1)}$ satisfies
	\begin{equation*}\label{3.27}
		\begin{split}
			&\partial_t \Delta_j \theta^{(n+1)}
			+(-\Delta)^{\frac{\alpha}{2}}\Delta_j\theta^{(n+1)}\\
			&\qquad
			=
			S_{n+4}\Delta_jF
			+[u^{(n)},\Delta_j] \cdot \nabla \theta^{(n+1)}
			-u^{(n)} \cdot \nabla \Delta_j \theta^{(n+1)},
		\end{split}
	\end{equation*}
	the same argument as in the derivation of (\ref{3.14}) yields that
	\begin{equation}\label{3.28}
		\begin{split}
			\|\Delta_j \theta^{(n+1)}(t)\|_{L^p}
			&\leqslant
				e^{-\lambda 2^{\alpha j}t}\|\Delta_j\theta^{(n)}(0)\|_{L^p}
				+C\int_0^t e^{-\lambda 2^{\alpha j}(t-\tau)}\|\Delta_jF(\tau)\|_{L^p} d\tau\\
			&\quad
				+\int_0^t e^{\lambda 2^{\alpha j}(t-\tau)}\|[u^{(n)}(\tau),\Delta_j] \cdot \nabla \theta^{(n+1)}(\tau)\|_{L^p} d\tau,
		\end{split}
	\end{equation}
	which implies that
	\begin{equation}\label{3.29}
		\begin{split}
			&2^{s_{\rm c}j}\|\Delta_j\theta^{n+1}(t)\|_{L^p}\\
			&\quad\leqslant
				C2^{s_{\rm c}j}\|\Delta_j \theta_0^{(n+1)}\|_{L^p}
				+C\int_0^t 2^{(1+\frac{2}{r}-\alpha)j}e^{-\lambda 2^{\alpha j}(t-\tau)}\|\Delta_jF(\tau)\|_{L^r} d\tau\\
			&\qquad
				+\int_0^t 2^{\alpha j}e^{-\lambda 2^{\alpha j}(t-\tau)} d\tau
				2^{(2s_{\rm c}-1-\frac{2}{p})j}\|[u^{(n)},\Delta_j] \cdot \nabla \theta^{(n+1)}\|_{L^{\infty}(0,T;L^p)}.
		\end{split}
	\end{equation}
	Taking $L^{\infty}_t(0,T)$-norm of (\ref{3.29}) and then taking $l^{q}(\mathbb{Z})$-norm,
	we see by Lemma\ref{t2-3} that
	\begin{equation}\label{3.31}
		\begin{split}
			&\|\theta^{(n+1)}\|_{\widetilde{L}^{\infty}(0,T;\dot{B}_{p,q}^{s_{\rm c}})}\\
			&\quad \leqslant
				C\|\theta_0^{(n+1)}\|_{\dot{B}_{p,q}^{s_{\rm c}}}
				+CT^{\frac{1}{\alpha}(2\alpha-1-\frac{2}{r})}\sup_{t>0}\|F(t)\|_{\dot{B}_{r,\infty}^0}\\
			&\qquad
				+C\left\|\left\{
					2^{(2s_{\rm c}-1-\frac{2}{p})j}\|[u^{(n)},\Delta_j] \cdot \nabla \theta^{(n+1)}\|_{L^{\infty}(0,T;L^p)}
				\right\}_{j\in \mathbb{Z}}\right\|_{l^{q}(\mathbb{Z})}\\
			&\quad \leqslant
				C\|\theta_0^{(n+1)}\|_{\dot{B}_{p,q}^{s_{\rm c}}}
				+CT^{\frac{1}{\alpha}(2\alpha-1-\frac{2}{r})}\sup_{t>0}\|F(t)\|_{\dot{B}_{r,\infty}^0}\\
			&\qquad
				+C\|\theta^{(n)}\|_{\widetilde{L}^{\infty}(0,T;\dot{B}_{p,q}^{s_{\rm c}})}
				\|\theta^{(n+1)}\|_{\widetilde{L}^{\infty}(0,T;\dot{B}_{p,q}^{s_{\rm c}})}.
		\end{split}
	\end{equation}
	On the other hand, by (\ref{3.28}), we see that
	\begin{equation}\label{3.32}
		\begin{split}
			&\|\Delta_j\theta^{n+1}(t)\|_{L^p}\\
			&\quad\leqslant
				C\|\Delta_j\theta_0^{(n+1)}\|_{L^p}
				+C\int_0^t 2^{(\frac{2}{r}-\frac{2}{p})j}e^{-\lambda 2^{\alpha j}(t-\tau)}d\tau
				\sup_{t>0}\|F(t)\|_{\dot{B}_{r,\infty}^0}\\
			&\qquad
				+\int_0^t 2^{(\alpha-s_{\rm c})j}e^{-\lambda 2^{\alpha j}(t-\tau)} d\tau\\
			&\qquad \quad
				\times
				\left\|\left\{
					2^{(2s_{\rm c}-1-\frac{2}{p})j}\|[u^{(n)},\Delta_j] \cdot \nabla \theta^{(n+1)}\|_{L^{\infty}(0,T;L^p)}
				\right\}_{j\in \mathbb{Z}}\right\|_{l^q(\mathbb{Z})}.
		\end{split}
	\end{equation}
	Taking $L^{\infty}_t(0,T)$-norm and then $l^1(\mathbb{Z})$-norm of (\ref{3.32}), we have
	\begin{equation*}\label{3.33}
		\begin{split}
			&\|\theta^{(n+1)}\|_{\widetilde{L}^{\infty}(0,T;\dot{B}_{p,1}^0)}\\
			&\quad\leqslant
				C\|\theta_0^{(n+1)}\|_{\dot{B}_{p,1}^0}
				+C\sum_{j\in \mathbb{Z}}\sup_{0\leqslant t\leqslant T}
				\int_0^t 2^{(1+\frac{2}{r}-\alpha)j}e^{-\lambda 2^{\alpha j}(t-\tau)}d\tau
				\sup_{t>0}\|F(t)\|_{\dot{B}_{r,\infty}^0}\\
			&\qquad
				+\sum_{j\in \mathbb{Z}}\sup_{0\leqslant t\leqslant T}
				\int_0^t 2^{(\alpha-s_{\rm c})j}e^{-\lambda 2^{\alpha j}(t-\tau)} d\tau\\
			&\qquad \qquad \qquad
				\times
				\left\|\left\{
					2^{(2s_{\rm c}-1-\frac{2}{p})j}\|[u^{(n)},\Delta_j] \cdot \nabla \theta^{(n+1)}\|_{L^{\infty}(0,T;L^p)}
				\right\}_{j\in \mathbb{Z}}\right\|_{l^q(\mathbb{Z})}.
		\end{split}
	\end{equation*}
	Using Lemma \ref{t2-3} and the second inequality of (\ref{2.13.8.1})
	we obtain that
	\begin{equation}\label{3.35}
		\begin{split}
			\|\theta^{(n+1)}\|_{\widetilde{L}^{\infty}(0,T;\dot{B}_{p,1}^0)}
			&\leqslant
				C\|\theta_0^{(n+1)}\|_{\dot{B}_{p,1}^0}
				+CT^{\frac{s_{\rm c}}{\alpha}}T^{\frac{1}{\alpha}(2\alpha-1-\frac{2}{r})}\sup_{t>0}\|F(t)\|_{\dot{B}_{r,\infty}^0}\\
			&\qquad
				+CT^{\frac{s_{\rm c}}{\alpha}}
				\|\theta^{(n)}\|_{\widetilde{L}^{\infty}(0,T;\dot{B}_{p,q}^{s_{\rm c}})}
				\|\theta^{(n+1)}\|_{\widetilde{L}^{\infty}(0,T;\dot{B}_{p,q}^{s_{\rm c}})}.
		\end{split}
	\end{equation}
	Hence, combining estimates (\ref{3.21}), (\ref{3.26}), (\ref{3.31}) and (\ref{3.35}), we obtain
	\begin{equation}\label{3.35.0}
		A_{n+1}
		\leqslant
		C_1 \sup_{t>0} \|F(t)\|_{\dot{B}_{r,\infty}^0}
		+C_1 A_{n-1}A_n
		+C_1 A_nA_{n+1}
	\end{equation}
	for some $C_1=C_1(\alpha,p,q,r,T)$.

	On the other hand,
	since $\theta^{(1)}$ satisfies
	\begin{equation*}\label{3.35.1}
		\begin{cases}
			\partial_t \theta^{(1)}+(-\Delta)^{\frac{\alpha}{2}}\theta^{(1)}=S_4F, \qquad & t>0,x\in \mathbb{R}^2,\\
			\theta^{(1)}(0,x)=0, \qquad & x\in \mathbb{R}^2,
		\end{cases}
	\end{equation*}
	the simpler argument than above yields that
	\begin{equation}\label{3.35.2}
		A_1\leqslant C_1\sup_{t>0}\|F(t)\|_{\dot{B}_{r,\infty}^0}.
	\end{equation}
	Hence, if $F$ satisfies
	\begin{equation*}\label{3.35.3}
		\sup_{t>0}\|F(t)\|_{\dot{B}_{r,\infty}^0}
		\leqslant \delta_1:=\frac{1}{8C_1^2},
	\end{equation*}
	then by (\ref{3.35.0}), (\ref{3.35.2}) and the inductive argument, we obtain
	\begin{equation*}\label{3.35.3.1}
		A_m\leqslant 2C_1\sup_{t>0}\|F(t)\|_{\dot{B}_{r,\infty}^0}
	\end{equation*}
	for all $m\in \mathbb{N}\cup \{0\}$.
	This completes the proof.
\end{proof}
Next lemma ensures the convergence of the approximation sequences.
\begin{lemm}\label{t3-2}
	There exists a positive constant $\delta_2=\delta_2(\alpha,p,q,r,\sigma,T)\leqslant \delta_1$
	such that if $F\in BC((0,\infty);\dot{B}_{r,\infty}^0(\mathbb{R}^2))$ satisfies
	\begin{equation*}
		\sup_{t>0}\|F(t)\|_{\dot{B}_{r,\infty}^0}\leqslant \delta_2,
	\end{equation*}
	then it holds
	\begin{equation*}\label{3.37}
		\sum_{n=0}^{\infty}\|\theta_0^{(n+1)}-\theta_0^{(n)}\|_{\dot{B}_{p,q}^{\sigma}}
		+\sum_{n=0}^{\infty}\|\theta^{(n+1)}-\theta^{(n)}\|_{\widetilde{L}^{\infty}(0,T;\dot{B}_{p,q}^{\sigma})}<\infty.
	\end{equation*}
\end{lemm}
\begin{proof}
	Due to
	\begin{equation*}
		\theta_0^{(n+2)}-\theta_0^{(n+1)}
		=\sum_{k=0}^{\infty} e^{-Tk(-\Delta)^{\frac{\alpha}{2}}}
		(\psi^{(n+1)}(T)-\psi^{(n)}(T))
	\end{equation*}
	and Lemma \ref{t2-5},
	we consider the estimates of $\psi^{(n+1)}(T)-\psi^{(n)}(T)$ in $\dot{B}_{p,q}^{\sigma-\alpha}(\mathbb{R}^2)\cap\dot{B}_{p,q}^{\sigma}(\mathbb{R}^2)$.
	Since $\psi^{(n+1)}-\psi^{(n)}$ satisfies
	\begin{equation}\label{3.47}
		\begin{split}
			&\partial_t (\psi^{(n+1)}-\psi^{(n)})
				+(-\Delta)^{\frac{\alpha}{2}} (\psi^{(n+1)}-\psi^{(n)})
				+u^{(n)} \cdot \nabla (\psi^{(n+1)}-\psi^{(n)})\\
			&\quad
				+u^{(n)} \cdot \nabla e^{-t(-\Delta)^{\frac{\alpha}{2}}}(\theta^{(n+1)}(0)-\theta^{(n)}(0))
				+(u^{(n)}-u^{(n-1)}) \cdot \nabla \theta^{(n)}
				=\Delta_{n+1}F,
		\end{split}
	\end{equation}
	we see that
	\begin{equation*}\label{3.47.0}
		\begin{split}
			&\partial_t \Delta_j(\psi^{(n+1)}-\psi^{(n)})
				+(-\Delta)^{\frac{\alpha}{2}} \Delta_j(\psi^{(n+1)}-\psi^{(n)})\\
			&\quad
				=\Delta_{n+1}\Delta_jF
				+[u^{(n)},\Delta_j]\cdot \nabla (\psi^{(n+1)}-\psi^{(n)})
				-u^{(n)} \cdot \nabla \Delta_j(\psi^{(n+1)}-\psi^{(n)})\\
			&\qquad
				-\Delta_j(u^{(n)} \cdot \nabla e^{-t(-\Delta)^{\frac{\alpha}{2}}}(\theta^{(n+1)}(0)-\theta^{(n)}(0)))
				-\Delta_j((u^{(n)}-u^{(n-1)}) \cdot \nabla \theta^{(n)}).
		\end{split}
	\end{equation*}
	Thus, the similar energy calculation as in the proof of Lemma \ref{t3-1} yields that
	\begin{equation*}\label{3.48}
		\begin{split}
			&\|\Delta_j(\psi^{(n+1)}(T)-\psi^{(n)}(T))\|_{L^p}\\
			&\quad \leqslant
			C\int_0^T e^{-\lambda 2^{\alpha j}(T-\tau)}2^{2(\frac{1}{r}-\frac{1}{p})j}\|\Delta_{n+1}\Delta_j F(\tau)\|_{L^r} d\tau\\
			&\qquad
				+\int_0^T e^{-\lambda 2^{\alpha j}(T-\tau)}\|[u^{(n)}(\tau),\Delta_j]\cdot \nabla (\psi^{(n+1)}(\tau)-\psi^{(n)}(\tau))\|_{L^p} d\tau\\
			&\qquad
				+\int_0^T e^{-\lambda 2^{\alpha j}(T-\tau)}
				\|\Delta_j(u^{(n)}(\tau) \cdot \nabla e^{-\tau(-\Delta)^{\frac{\alpha}{2}}}(\theta^{(n+1)}(0)-\theta^{(n)}(0)))\|_{L^p} d\tau\\
			&\qquad
				+\int_0^T e^{-\lambda 2^{\alpha j}(T-\tau)}\|\Delta_j((u^{(n)}(\tau)-u^{(n-1)}(\tau)) \cdot \nabla \theta^{(n)}(\tau))\|_{L^p} d\tau.
		\end{split}
	\end{equation*}
	Let $s\in \{\sigma,\sigma-\alpha\}$.
	Multiplying this by $2^{sj}$
	and
	taking $l^q(\mathbb{Z})$-norm of this, we obtain that
	\begin{equation*}\label{3.49}
		\begin{split}
			&\|\psi^{(n+1)}(T)-\psi^{(n)}(T)\|_{\dot{B}_{p,q}^s}\\
			&\quad \leqslant
			CT^{\frac{1}{\alpha}(2\alpha-1-\frac{2}{r})}2^{-(s_{\rm c}-\sigma)n}\sup_{t>0}\|F(t)\|_{\dot{B}_{r,\infty}^0}\\
			&\qquad
				+CT^{\frac{\sigma-s}{\alpha}}
				\left\|\left\{
					2^{(s_{\rm c}+(\sigma-1)-\frac{2}{p})j}
					\|[u^{(n)},\Delta_j]\cdot \nabla (\psi^{(n+1)}-\psi^{(n)})\|_{L^{\infty}(0,T;L^p)}
				\right\}_{j\in \mathbb{Z}}\right\|_{l^q(\mathbb{Z})}\\
			&\qquad
				+\left\|\left\{
					\int_0^T 2^{(\alpha+s-\sigma)j} e^{-\lambda 2^{\alpha j}(T-\tau)}2^{(s_{\rm c}+(\sigma-1)-\frac{2}{p})j}
				\right.\right.\\
			&\qquad\qquad\quad
				\left.\left.
					\times
					\|\Delta_j(u^{(n)} \cdot \nabla e^{-t(-\Delta)^{\frac{\alpha}{2}}}(\theta^{(n+1)}(0)-\theta^{(n)}(0)))\|_{L^p} d\tau
				\right\}_{j\in \mathbb{Z}}\right\|_{l^q(\mathbb{Z})}\\
			&\qquad
				+CT^{\frac{\sigma-s}{\alpha}}\|(u^{(n)}-u^{(n-1)}) \cdot \nabla \theta^{(n)}\|_
				{\widetilde{L}^{\infty}(0,T;\dot{B}_{p,q}^{\sigma+(s_{\rm c}-1)-\frac{2}{p}})}.
		\end{split}
	\end{equation*}
	Therefore, it follows from Lemmas \ref{t2-2}, \ref{t2-3} and (\ref{2.13.0}) that
	\begin{equation}\label{3.49}
		\begin{split}
			&\|\psi^{(n+1)}(T)-\psi^{(n)}(T)\|_{\dot{B}_{p,q}^s}\\
			&\quad \leqslant
			CT^{\frac{\sigma-s}{\alpha}}T^{\frac{1}{\alpha}(2\alpha-1-\frac{2}{r})}2^{-(s_{\rm c}-\sigma)n}
			\sup_{t>0}\|F(t)\|_{\dot{B}_{r,\infty}^0}\\
			&\qquad
				+CT^{\frac{\sigma-s}{\alpha}}
				\|\theta^{(n)}\|_{\widetilde{L}^{\infty}(0,T;\dot{B}_{p,q}^{s_{\rm c}})}
				\|\psi^{(n+1)}-\psi^{(n)}\|_{\widetilde{L}^{\infty}(0,T;\dot{B}_{p,q}^{\sigma})}\\
			&\qquad
				+CT^{\frac{\sigma-s}{\alpha}}
				(1+T^{-\frac{1-\alpha}{\alpha}})
				\|\theta^{(n)}\|_{X_T^{p,q}}
				\|\theta^{(n)}(0)-\theta^{(n-1)}(0)\|_{\dot{B}_{p,q}^{\sigma}}\\
			&\qquad
			+CT^{\frac{\sigma-s}{\alpha}}
			\|\theta^{(n)}\|_{\widetilde{L}^{\infty}(0,T;\dot{B}_{p,q}^{s_{\rm c}})}
			\|\theta^{(n)}-\theta^{(n-1)}\|_{\widetilde{L}^{\infty}(0,T;\dot{B}_{p,q}^{\sigma})}.
		\end{split}
	\end{equation}
	This gives
	\begin{equation}\label{3.49.1}
		\begin{split}
			&\|\theta_0^{(n+1)}-\theta_0^{(n)}\|_{\dot{B}_{p,q}^{\sigma}}\\
			&\quad\leqslant
			C(T^{-1}\|\psi^{(n+1)}(T)-\psi^{(n)}(T)\|_{\dot{B}_{p,q}^{\sigma-\alpha}}+\|\psi^{(n+1)}(T)-\psi^{(n)}(T)\|_{\dot{B}_{p,q}^{\sigma}})\\
			&\quad \leqslant
			CT^{\frac{1}{\alpha}(2\alpha-1-\frac{2}{r})}2^{-(s_{\rm c}-\sigma)n}
			\sup_{t>0}\|F(t)\|_{\dot{B}_{r,\infty}^0}\\
			&\qquad
				+CA_nB_n
				+C(1+T^{-\frac{1-\alpha}{\alpha}})A_nB_{n-1}.\\
		\end{split}
	\end{equation}
	Since $\theta^{(n+2)}-\theta^{(n+1)}$ satisfy
	\begin{equation*}\label{3.38}
		\begin{split}
			&
			\partial_t (\theta^{(n+2)}- \theta^{(n+1)})
			+ (-\Delta)^{\frac{\alpha}{2}}(\theta^{(n+2)}- \theta^{(n+1)})\\
			&\qquad
			+ u^{(n+1)} \cdot \nabla (\theta^{(n+2)}- \theta^{(n+1)})
			+ (u^{(n+1)}-u^{(n)}) \cdot \nabla \theta^{(n+1)}
			= \Delta_{n+2}F,
		\end{split}
	\end{equation*}
	we see that
	\begin{equation*}\label{3.39}
		\begin{split}
			&
			\partial_t \Delta_j(\theta^{(n+2)}- \theta^{(n+1)})
			+ (-\Delta)^{\frac{\alpha}{2}}\Delta_j(\theta^{(n+2)}- \theta^{(n+1)})\\
			&\quad= \Delta_j\Delta_{n+2}F
			+ [u^{(n+1)}, \Delta_j ]\cdot \nabla (\theta^{(n+2)}- \theta^{(n+1)})\\
			&\qquad- u^{(n+1)} \cdot \nabla \Delta_j(\theta^{(n+2)}-\theta^{(n+1)})
			- \Delta_j((u^{(n+1)}-u^{(n)}) \cdot \nabla \theta^{(n+1)}).
		\end{split}
	\end{equation*}
	By the similar energy calculation as in the proof of Lemma \ref{t3-1}, we have
	\begin{equation}\label{3.40}
		\begin{split}
			&
			\|\Delta_j(\theta^{(n+2)}(t)- \theta^{(n+1)}(t))\|_{L^p}\\
			&\qquad
				\leqslant
				e^{-\lambda 2^{\alpha j}t}\|\Delta_j(\theta^{(n+2)}(0)- \theta^{(n+1)}(0))\|_{L^p}\\
			&\qquad \quad
				+C\int_0^t 2^{(\frac{2}{r}-\frac{2}{p})j}e^{-\lambda 2^{\alpha j}(t-\tau)} \|\Delta_j\Delta_{n+2}F(\tau)\|_{L^r} d\tau\\
			&\qquad \quad
				+\int_0^t e^{-\lambda 2^{\alpha j}(t-\tau)}
				\|[u^{(n+1)}(\tau), \Delta_j ]\cdot \nabla (\theta^{(n+2)}(\tau)-\theta^{(n+1)}(\tau))\|_{L^p} d\tau\\
			&\qquad \quad
				+\int_0^t e^{-\lambda 2^{\alpha j}(t-\tau)} \|\Delta_j((u^{(n+1)}(\tau)-u^{(n)}(\tau)) \cdot \nabla \theta^{(n+1)}(\tau))\|_{L^p} d\tau.
		\end{split}
	\end{equation}
	Multiplying (\ref{3.40}) by $2^{\sigma j}$, we see that
	\begin{equation*}\label{3.41}
		\begin{split}
			&
			2^{\sigma j}\|\Delta_j(\theta^{(n+2)}(t)- \theta^{(n+1)}(t))\|_{L^p}\\
			&\leqslant
				2^{\sigma j}\|\Delta_j(\theta^{(n+2)}(0)- \theta^{(n+1)}(0))\|_{L^p}\\
			&+C\int_0^t 2^{(1+\frac{2}{r}-\alpha)j}e^{-\lambda 2^{\alpha j}(t-\tau)}d\tau
				\sup_{t>0}\|\Delta_{n+2}F(t)\|_{\dot{B}_{r,\infty}^{(s_{\rm c}-\sigma)}}\\
			&+\int_0^t 2^{\alpha j}e^{-\lambda 2^{\alpha j}(t-\tau)} d\tau
				2^{(s_{\rm c}+(\sigma-1)-\frac{2}{p})j}
				\|[u^{(n+1)}, \Delta_j ]\cdot \nabla (\theta^{(n+2)}-\theta^{(n+1)})\|_{L^{\infty}(0,T;L^p)} \\
			&+\int_0^t 2^{\alpha j}e^{-\lambda 2^{\alpha j}(t-\tau)} d\tau
				2^{(\sigma+(s_{\rm c}-1)-\frac{2}{p})j}
				\|\Delta_j((u^{(n+1)}-u^{(n)}) \cdot \nabla \theta^{(n+1)})\|_{L^{\infty}(0,T;L^p)}.
		\end{split}
	\end{equation*}
	By taking $L^{\infty}_t(0,T)$ and then $l^q(\mathbb{Z})$-norm,
	it follows from (\ref{2.13.8.1}) and Lemmas \ref{t2-2} and \ref{t2-3} that
	\begin{equation}\label{3.43}
		\begin{split}
			&
			\|\theta^{(n+2)}- \theta^{(n+1)}\|_{\widetilde{L}^{\infty}(0,T;\dot{B}_{p,q}^{\sigma})}\\
			&\qquad\quad
				\leqslant
				\|\theta^{(n+2)}(0)- \theta^{(n+1)}(0)\|_{\dot{B}_{p,q}^{\sigma}}\\
			&\qquad\qquad
				+CT^{\frac{1}{\alpha}(2\alpha-1-\frac{2}{r})}
				2^{-(s_{\rm c}-\sigma)n}
				\sup_{t>0}\|F(t)\|_{\dot{B}_{r,\infty}^0}\\
			&\qquad\qquad
				+C\|\theta^{(n+1)}\|_{\widetilde{L}^{\infty}(0,T;\dot{B}_{p,q}^{s_{\rm c}})}
				\|\theta^{(n+2)}- \theta^{(n+1)}\|_{\widetilde{L}^{\infty}(0,T;\dot{B}_{p,q}^{\sigma})} \\
			&\qquad\qquad
				+C\|\theta^{(n+1)}\|_{\widetilde{L}^{\infty}(0,T;\dot{B}_{p,q}^{s_{\rm c}})}
				\|\theta^{(n+1)}- \theta^{(n)}\|_{\widetilde{L}^{\infty}(0,T;\dot{B}_{p,q}^{\sigma})}.
		\end{split}
	\end{equation}
	Since it holds
	\begin{equation*}\label{3.44}
		\begin{split}
			\|\theta^{(n+2)}(0)- \theta^{(n+1)}(0)\|_{\dot{B}_{p,q}^{\sigma}}
			&\leqslant
				\|S_{n+5}(\theta_0^{(n+2)}-\theta_0^{(n+1)})\|_{\dot{B}_{p,q}^{\sigma}}
				+\|\Delta_{n+2} \theta^{(n+1)}\|_{\dot{B}_{p,q}^{\sigma}}\\
			&\leqslant
				C\|\theta_0^{(n+2)}-\theta_0^{(n+1)}\|_{\dot{B}_{p,q}^{\sigma}}
				+C 2^{-(s_{\rm c}-\sigma)n}
				\|\theta_0^{(n+1)}\|_{\dot{B}_{p,q}^{s_{\rm c}}}\\
			&\leqslant
				C\|\theta_0^{(n+2)}-\theta_0^{(n+1)}\|_{\dot{B}_{p,q}^{\sigma}}
				+C2^{-(s_{\rm c}-\sigma)n}A_{n+1},
		\end{split}
	\end{equation*}
	we have by (\ref{3.43}) that
	\begin{equation}\label{3.45}
		\begin{split}
			&\|\theta^{(n+2)}-\theta^{(n+1)}\|_{\widetilde{L}^{\infty}(0,T;\dot{B}_{p,q}^{\sigma})}\\
			&\qquad\leqslant
				C\|\theta_0^{(n+2)}-\theta_0^{(n+1)}\|_{\dot{B}_{p,q}^{\sigma}}
				+C 2^{-(s_{\rm c}-\sigma)n}
				\left(
					T^{\frac{1}{\alpha}(2\alpha-1-\frac{2}{r})}\sup_{t>0}\|F(t)\|_{\dot{B}_{r,\infty}^0}
					+A_{n+1}
				\right) \\
			&\qquad\quad +C \|\theta^{(n+1)}\|_{X_T^{p,q}}\|\theta^{(n+1)}-\theta^{(n)}\|_{\widetilde{L}^{\infty}(0,T;\dot{B}_{p,q}^{\sigma})}\\
			&\qquad\quad +C \|\theta^{(n+1)}\|_{X_T^{p,q}}\|\theta^{(n+2)}-\theta^{(n+1)}\|_{\widetilde{L}^{\infty}(0,T;\dot{B}_{p,q}^{\sigma})}.
		\end{split}
	\end{equation}
	Therefore, combining (\ref{3.49.1}) and (\ref{3.45}), we obtain
	\begin{equation}\label{3.45.1}
		\begin{split}
			B_{n+1}
			&\leqslant
			C_2 2^{-(s_{\rm c}-\sigma)n}
			\left(
				\sup_{t>0}\|F(t)\|_{\dot{B}_{r,\infty}^0}
				+A_{n+1}
			\right) \\
			&\quad
				+C_2 A_nB_{n-1}
				+C_2 (A_n+A_{n+1})B_n
				+C_2A_{n+1}B_{n+1}
		\end{split}
	\end{equation}
	for some $C_2=C_2(\alpha,p,q,r,\sigma,T)>0$.
	Here, we assume that
	\begin{equation*}
		\sup_{t>0}\|F(t)\|_{\dot{B}_{r,\infty}^0}\leqslant \delta_2
		=:\min \left\{ \delta_1,\frac{1}{16C_1C_2} \right\}.
	\end{equation*}
	Let $N\in \mathbb{N}$ satisfy $N\geqslant 2$.
	Then, summing (\ref{3.45.1}) over $n=1,...,N-1$ and using Lemma \ref{t3-1}, we have
	\begin{equation*}\label{3.62}
		\begin{split}
			\sum_{n=1}^{N-1}B_{n+1}
			&\leqslant
				C_T\delta_1
				\sum_{n=1}^{N-1}2^{-(s_{\rm c}-\sigma)n}\\
			&\quad
				+2C_1C_2 \sum_{n=1}^{N-1}B_{n-1}
				+4C_1C_2
				\sum_{n=1}^{N-1}B_n
				+2C_1C_2 \sum_{n=1}^{N-1}B_{n+1}
		\end{split}
	\end{equation*}
	for some constant $C_T>0$ depending on $T$.
	This implies
	\begin{equation*}\label{3.63}
		\begin{split}
			\sum_{n=2}^{N}B_n
			&\leqslant
				C_T\delta_1
				\sum_{n=1}^{N-1}2^{-(s_{\rm c}-\sigma)n}
				+\frac{1}{8}\sum_{n=0}^{N-2}B_n
				+\frac{2}{8}
				\sum_{n=1}^{N-1}B_n
				+\frac{1}{8}\sum_{n=2}^{N}B_n\\
			&\leqslant
				C_T\delta_1
				\sum_{n=1}^{\infty}2^{-(s_{\rm c}-\sigma)n}
				+\frac{1}{2}\sum_{n=0}^{N}B_n.\\
		\end{split}
	\end{equation*}
	Hence, we have
	\begin{equation*}\label{3.65}
		\frac{1}{2}\sum_{n=0}^{\infty}B_n
		\leqslant
		C_T\delta_1\sum_{n=1}^{\infty}2^{-(s_{\rm c}-\sigma)n}+B_0+B_1
		<\infty,
	\end{equation*}
	which completes the proof.
\end{proof}
\begin{lemm}\label{t3-3}
	There exists a positive constant $C_3=C_3(\alpha,p,q,r,T,\sigma)$
	such that
	\begin{equation}\label{3.50}
		\begin{split}
			&\|\theta-\widetilde{\theta}\|_{\widetilde{L}^{\infty}(0,T;\dot{B}_{p,q}^{\sigma})}\\
			&\quad\leqslant
				C_3
				\left(
					\|\theta\|_{\widetilde{L}^{\infty}(0,T;\dot{B}_{p,q}^{s_{\rm c}})}
					+\|\widetilde{\theta}\|_{\widetilde{L}^{\infty}(0,T;\dot{B}_{p,q}^{s_{\rm c}})}
				\right)
				\|\theta-\widetilde{\theta}\|_{\widetilde{L}^{\infty}(0,T;\dot{B}_{p,q}^{\sigma})}
		\end{split}
	\end{equation}
	for all $T$-time periodic solutions
	$\theta\in BC([0,\infty);B_{p,q}^{s_{\rm c}}(\mathbb{R}^2))\cap X_{T}^{p,q}$ and
	$\widetilde{\theta}\in BC([0,\infty);B_{p,q}^{s_{\rm c}}(\mathbb{R}^2))\cap \widetilde{L}^{\infty}(0,\infty;\dot{B}_{p,q}^{s_{\rm c}}(\mathbb{R}^2))$
	to (\ref{QG}) with the same $T$-time periodic external force $F$.
\end{lemm}
\begin{proof}
	Since $\theta-\widetilde{\theta}$ satisfies
	\begin{equation*}\label{3.51}
		\partial_t(\theta-\widetilde{\theta})
		+(-\Delta)^{\frac{\alpha}{2}}(\theta-\widetilde{\theta})
		+u \cdot \nabla(\theta-\widetilde{\theta})
		+(u-\widetilde{u}) \cdot \nabla \widetilde{\theta}=0,
	\end{equation*}
	where $u=\mathcal{R}^{\perp}\theta$, $\widetilde{u}=\mathcal{R}^{\perp}\widetilde{\theta}$,
	we see that
	\begin{equation*}
		\begin{split}
			&\partial_t\Delta_j(\theta-\widetilde{\theta})
			+(-\Delta)^{\frac{\alpha}{2}}\Delta_j(\theta-\widetilde{\theta})\\
			&\qquad=[u,\Delta_j]\cdot \nabla(\theta-\widetilde{\theta})
			-u\cdot \nabla\Delta_j(\theta-\widetilde{\theta})
			-\Delta_j ((u-\widetilde{u}) \cdot \nabla \widetilde{\theta}).
		\end{split}
	\end{equation*}
	Therefore, it follows from the similar energy calculation as in the derivation of (\ref{3.45}) that
	\begin{equation}\label{3.52}
		\begin{split}
			&\|\theta-\widetilde{\theta}\|_{\widetilde{L}^{\infty}(0,T;\dot{B}_{p,q}^{\sigma})}\\
			&\qquad\leqslant
			\|\theta(0)-\widetilde{\theta}(0)\|_{\dot{B}_{p,q}^{\sigma}}\\
			&\qquad \quad
			+C\left(
				\|\theta\|_{\widetilde{L}^{\infty}(0,T;\dot{B}_{p,q}^{s_{\rm c}})}
				+\|\widetilde{\theta}\|_{\widetilde{L}^{\infty}(0,T;\dot{B}_{p,q}^{s_{\rm c}})}
			\right)
			\|\theta-\widetilde{\theta}\|_{\widetilde{L}^{\infty}(0,T;\dot{B}_{p,q}^{\sigma})}.
		\end{split}
	\end{equation}
	Next, we derive the estimate for $\theta(0)-\widetilde{\theta}(0)$.
	Since $\theta-\widetilde{\theta}$ is $T$-time periodic and the Duhamel principle gives
	\begin{equation*}\label{3.53}
		\theta(t)-\widetilde{\theta}(t)
		=e^{-t(-\Delta)^{\frac{\alpha}{2}}}(\theta(0)-\widetilde{\theta}(0))
		-\int_0^t e^{-(t-\tau)(-\Delta)^{\frac{\alpha}{2}}}(u(\tau) \cdot \nabla \theta(\tau)-\widetilde{u}(\tau) \cdot \nabla \widetilde{\theta}(\tau)) d\tau,
	\end{equation*}
	we have by Lemma \ref{t2-6} that
	\begin{equation}\label{3.54}
		\|\theta(0)-\widetilde{\theta}(0)\|_{\dot{B}_{p,q}^{\sigma}}
		\leqslant
		C(T^{-1}\|\psi(T)-\widetilde{\psi}(T)\|_{\dot{B}_{p,q}^{\sigma-\alpha}}
		+\|\psi(T)-\widetilde{\psi}(T)\|_{\dot{B}_{p,q}^{\sigma}}),
	\end{equation}
	where $\psi(t):=\theta(t)-e^{-t(-\Delta)^{\frac{\alpha}{2}}}\theta(0)$
	and $\widetilde{\psi}(t):=\widetilde{\theta}(t)-e^{-t(-\Delta)^{\frac{\alpha}{2}}}\widetilde{\theta}(0)$.
	Since it holds
	\begin{equation}\label{3.56}
		\begin{split}
			&\partial_t(\psi-\widetilde{\psi})+(-\Delta)^{\frac{\alpha}{2}}(\psi-\widetilde{\psi})\\
			&\qquad=-u\cdot \nabla(\psi-\widetilde{\psi})-u \cdot \nabla e^{-t(-\Delta)^{\frac{\alpha}{2}}}(\theta(0)-\widetilde{\theta}(0))
			-(u-\widetilde{u}) \cdot \nabla \widetilde{\theta},
		\end{split}
	\end{equation}
	applying $\Delta_j$ to (\ref{3.56}), we have
	\begin{equation*}\label{3.57}
		\begin{split}
			&\partial_t\Delta_j(\psi-\widetilde{\psi})+(-\Delta)^{\frac{\alpha}{2}}\Delta_j(\psi-\widetilde{\psi})\\\
			&\qquad=[u,\Delta_j]\cdot \nabla(\psi-\widetilde{\psi})-u \cdot \nabla \Delta_j(\psi-\widetilde{\psi})\\
			&\qquad \quad -\Delta_j (u \cdot \nabla e^{-t(-\Delta)^{\frac{\alpha}{2}}}(\theta(0)-\widetilde{\theta}(0)))
			-\Delta_j ((u-\widetilde{u}) \cdot \nabla \widetilde{\theta}).
		\end{split}
	\end{equation*}
	Hence, by the similar energy calculation as in the derivation of (\ref{3.49}), we obtain
	\begin{equation}\label{3.58}
		\begin{split}
			\begin{split}
				&T^{-1}\|\psi(T)-\widetilde{\psi}(T)\|_{\dot{B}_{p,q}^{\sigma-\alpha}}
				+\|\psi(T)-\widetilde{\psi}(T)\|_{\dot{B}_{p,q}^{\sigma}}\\
				&\qquad\leqslant
				C(1+T^{-\frac{1-\alpha}{\alpha}})\left(\|\theta\|_{X_{T}^{p,q}}
				+\|\widetilde{\theta}\|_{\widetilde{L}^{\infty}(0,T;\dot{B}_{p,q}^{s_{\rm c}})}\right)
				\|\theta-\widetilde{\theta}\|_{\widetilde{L}^{\infty}(0,T;\dot{B}_{p,q}^{\sigma})}.
			\end{split}
		\end{split}
	\end{equation}
	Here, we have used
	\begin{equation*}\label{3.58.1}
		\|\theta(0)-\widetilde{\theta}(0)\|_{\dot{B}_{p,q}^{\sigma}}
		\leqslant \sup_{0\leqslant t\leqslant T}\|\theta(t)-\widetilde{\theta}(t)\|_{\dot{B}_{p,q}^{\sigma}}
		\leqslant \|\theta-\widetilde{\theta}\|_{\widetilde{L}^{\infty}(0,T;\dot{B}_{p,q}^{\sigma})}.
	\end{equation*}
	Combining (\ref{3.52}), (\ref{3.54}) and (\ref{3.58}), we get (\ref{3.50}).
	This completes the proof.
\end{proof}
Now we are in a position to prove Theorem \ref{t1-1}.
\begin{proof}[Proof of Theorem \ref{t1-1}]
	Let $\alpha,p,q,r$ and $T$ satisfy the assumptions of Theorem \ref{t1-1} and let $\sigma:=\alpha/2$.
	Then, $\sigma$ satisfies $\alpha-2/p<\sigma<2/p$.
	We put
	\begin{equation*}\label{3.59}
		\delta
		:=\min
		\left\{
			\delta_1,\delta_2,\frac{1}{8C_1C_3}
		\right\}
	\end{equation*}
	and let $F\in BC((0,\infty);\dot{B}_{r,\infty}^0(\mathbb{R}^2))$ satisfy
	\begin{equation*}\label{3.59.1}
		\sup_{t>0}\|F(t)\|_{\dot{B}_{r,\infty}^0}\leqslant \delta.
	\end{equation*}
	It follows from (\ref{3.5}) that
	\begin{equation}\label{3.61}
		\sup_{n\in \mathbb{N}\cup\{0\}}\|\theta^{(n)}\|_{X_T^{p,q}}
		\leqslant 2C_1\sup_{t>0}\|F(t)\|_{\dot{B}_{r,\infty}^0}
		\leqslant 2C_1\delta=:K.
	\end{equation}
	From Lemma \ref{t3-2}, there exist limits $\theta_0\in \dot{B}_{p,q}^{\sigma}(\mathbb{R}^2)$ and $\theta\in L^{\infty}(0,T;\dot{B}_{p,q}^{\sigma}(\mathbb{R}^2))$ such that
	\begin{equation*}\label{3.66}
		\begin{split}
			\theta_0&=\sum_{n=0}^{\infty}(\theta^{(n+1)}_0-\theta_0^{(n)})=\lim_{n\to \infty}\theta_0^{(n)}
			\qquad {\rm in }\ \dot{B}_{p,q}^{\sigma}(\mathbb{R}^2),\\
			\theta&=\sum_{n=0}^{\infty}(\theta^{(n+1)}-\theta^{(n)})=\lim_{n\to \infty}\theta^{(n)}
			\qquad {\rm in }\ L^{\infty}(0,T;\dot{B}_{p,q}^{\sigma}(\mathbb{R}^2)).
		\end{split}
	\end{equation*}
	By Lemma \ref{t3-1} and (\ref{3.61}), we see that $\theta_0\in \dot{B}_{p,q}^{s_{\rm c}}(\mathbb{R}^2)$, $\theta\in X_T^{p,q}$ and
	\begin{equation}\label{3.66.1}
		\|\theta\|_{\widetilde{L}^{\infty}(0,T;\dot{B}_{p,q}^{s_{\rm c}})}\leqslant \|\theta\|_{X_T^{p,q}}\leqslant K.
	\end{equation}
	It is easy to check that $\theta$ is a solution to (\ref{QG}) on $[0,T]$.
	Next, we show the continuity in time of the solution $\theta$ by the idea in \cite{Danchin}.
	Let $(s,\rho)\in \{(s_{\rm c},q),(0,1)\}$.
	Since it holds $\partial_t\Delta_j\theta=\Delta_jF-(-\Delta)^{\frac{\alpha}{2}}\Delta_j\theta-\Delta_j(u\cdot \nabla \theta)$, we have
	\begin{equation*}\label{3.73}
		\begin{split}
			&\|\partial_t\Delta_j\theta\|_{\dot{B}_{p,\rho}^s}\\
			&\quad\leqslant
			C2^{sj}\|\partial_t\Delta_j\theta\|_{L^p}\\
			&\quad\leqslant
			C2^{sj+2(\frac{1}{r}-\frac{1}{p})j}\|\Delta_jF(t)\|_{L^r}
			+C2^{sj+\alpha j}\|\Delta_j\theta\|_{L^p}
			+C2^{sj+j}\|u(t)\|_{L^p}\|\theta(t)\|_{L^{\infty}}\\
			&\quad\leqslant
			C2^{sj+2(\frac{1}{r}-\frac{1}{p})j}\sup_{t>0}\|F(t)\|_{\dot{B}_{r,\infty}^0}
			+C2^{sj+\alpha j}\|\theta\|_{\widetilde{L}^{\infty}(0,T;\dot{B}_{p,1}^0)}
			+C2^{sj+j}\|\theta\|_{X_T^{p,q}}^2,
		\end{split}
	\end{equation*}
	which implies $\partial_t\Delta_j\theta\in L^{\infty}(0,T;\dot{B}_{p,\rho}^s(\mathbb{R}^2))$.
	Therefore, we have
	\begin{equation*}\label{3.74}
		\Theta_m:=\sum_{|j|\leqslant m}\Delta_j \theta\in C([0,T];\dot{B}_{p,1}^0(\mathbb{R}^2)\cap \dot{B}_{p,q}^{s_{\rm c}}(\mathbb{R}^2)),
		\qquad	m\in \mathbb{N}.
	\end{equation*}
	It follows from $q<\infty$ and (\ref{3.66.1}) that
	\begin{equation*}\label{3.75}
		\begin{split}
			&\|\Theta_m-\theta\|_{L^{\infty}(0,T;\dot{B}_{p,1}^0\cap \dot{B}_{p,q}^{s_{\rm c}})}\\
			&\quad\leqslant
			C\sum_{|j|\geqslant m}\|\Delta_j\theta\|_{L^{\infty}(0,T;L^p)}
			+C\left\|\left\{2^{s_{\rm c}j}\|\Delta_j\theta\|_{L^{\infty}(0,T;L^p)}\right\}_{\{j:|j|\geqslant m\}}\right\|_{l^q}\\
			&\quad\to 0,
			\qquad{\rm as}\ m\to \infty.
		\end{split}
	\end{equation*}
	Hence, we see that
	$\theta
	\in C([0,T];\dot{B}_{p,1}^0(\mathbb{R}^2)\cap \dot{B}_{p,q}^{s_{\rm c}}(\mathbb{R}^2))
	\subset
	C([0,T];B_{p,q}^{s_{\rm c}}(\mathbb{R}^2))$.
	Since
	\begin{equation*}\label{3.69}
		\begin{split}
			&\|\theta^{(n)}(0)-\theta_0\|_{\dot{B}_{p,q}^{\sigma}}\\
			&\quad\leqslant C\|S_{n+3}(\theta_0^{(n)}-\theta_0)\|_{\dot{B}_{p,q}^{\sigma}}
			+\|(1-S_{n+3})\theta_0\|_{\dot{B}_{p,q}^{\sigma}}\\
			&\quad\leqslant C\|\theta_0^{(n)}-\theta_0\|_{\dot{B}_{p,q}^{\sigma}}
			+\|(1-S_{n+3})\theta_0\|_{\dot{B}_{p,q}^{\sigma}}
			\to 0,\\
			&\|(1-e^{-T(-\Delta)^{\frac{\alpha}{2}}})\theta_0^{(n+1)}
				-(1-e^{-T(-\Delta)^{\frac{\alpha}{2}}})\theta_0\|_{\dot{B}_{p,q}^{\sigma}}\\
			&\quad\leqslant
			\|\theta_0^{(n+1)}-\theta_0\|_{\dot{B}_{p,q}^{\sigma}}
			+\|e^{-T(-\Delta)^{\frac{\alpha}{2}}}(\theta_0^{(n+1)}-\theta_0)\|_{\dot{B}_{p,q}^{\sigma}}\\
			&\quad\leqslant
			C\|\theta_0^{(n+1)}-\theta_0\|_{\dot{B}_{p,q}^{\sigma}}\to 0,\\
			&\|(\theta^{(n)}(T)-e^{-(-\Delta)^{\frac{\alpha}{2}}}\theta^{n}(0))-(\theta(T)-e^{-(-\Delta)^{\frac{\alpha}{2}}}\theta_0)\|_{\dot{B}_{p,q}^{\sigma}}\\
			&\quad\leqslant
			\|\theta^{(n)}(T)-\theta(T)\|_{\dot{B}_{p,q}^{\sigma}}+\|e^{-T(-\Delta)^{\frac{\alpha}{2}}}(\theta^{(n)}(0)-\theta_0)\|_{\dot{B}_{p,q}^{\sigma}}\\
			&\quad\leqslant
			\|\theta^{(n)}-\theta\|_{L^{\infty}(0,T;\dot{B}_{p,q}^{\sigma})}
			+C\|\theta^{(n)}(0)-\theta_0\|_{\dot{B}_{p,q}^{\sigma}}
			\to 0
		\end{split}
	\end{equation*}
	as $n\to \infty$,
	we obtain by letting $n\to \infty$ in (\ref{3.3.1}) that
	\begin{equation*}\label{3.71}
		(1-e^{-T(-\Delta)^{\frac{\alpha}{2}}})\theta_0=\theta(T)-e^{-T(-\Delta)^{\frac{\alpha}{2}}}\theta_0,\qquad \theta(0)=\theta_0,
	\end{equation*}
	which implies
	\begin{equation}\label{3.72}
		\theta(T)=\theta(0)=\theta_0.
	\end{equation}
	Let us extend $\theta$ to the function on the interval $[0,\infty)$ periodically as
	\begin{equation*}\label{3.76}
		\theta(t)=\theta(t-NT),\qquad {\rm for}\ NT<t\leqslant (N+1)T,\ N\in \mathbb{N}.
	\end{equation*}
	Then, $\theta\in BC([0,\infty);B_{p,q}^{s_{\rm c}}(\mathbb{R}^2))$
	and $\theta$ is a $T$-time periodic solution to (\ref{QG}) satisfying (\ref{1.4}).
	Finally, we prove the uniqueness.
	Let $\widetilde{\theta}$ be arbitrary solution satisfying (\ref{1.4}).
	Note that since $0<\sigma<s_c$,
	we see that
	$\theta,\widetilde{\theta} \in L^{\infty}(0,T;L^p(\mathbb{R}^2))
	\cap\widetilde{L}^{\infty}(0,T;\dot{B}_{p,q}^{s_{\rm c}}(\mathbb{R}^2))
	\subset \widetilde{L}^{\infty}(0,T;\dot{B}_{p,q}^{\sigma}(\mathbb{R}^2))$
	holds by the similar calculation as (\ref{2.13.2.0}).
	Then, it follows from (\ref{3.50}) and (\ref{3.66.1}) that
	\begin{equation*}\label{3.77}
		\begin{split}
			\|\theta-\widetilde{\theta}\|_{\widetilde{L}^{\infty}(0,T;\dot{B}_{p,q}^{\sigma})}
			&\leqslant
				C_3
				\left(
					\|\theta\|_{X_T^{p,q}}
					+\|\widetilde{\theta}\|_{\widetilde{L}^{\infty}(0,T;\dot{B}_{p,q}^{s_{\rm c}})}
				\right)
				\|\theta-\widetilde{\theta}\|_{\widetilde{L}^{\infty}(0,T;\dot{B}_{p,q}^{\sigma})}\\
			&\leqslant
				2KC_3
				\|\theta-\widetilde{\theta}\|_{\widetilde{L}^{\infty}(0,T;\dot{B}_{p,q}^{\sigma})}\\
			&\leqslant
				\frac{1}{2}
				\|\theta-\widetilde{\theta}\|_{\widetilde{L}^{\infty}(0,T;\dot{B}_{p,q}^{\sigma})}.
		\end{split}
	\end{equation*}
	Thus, we see that $\widetilde{\theta}=\theta$ on $[0,T]$.
	The periodicity of $\theta$ and $\widetilde{\theta}$ implies $\theta=\widetilde{\theta}$ on $[0,\infty)$.
	This completes the proof.
\end{proof}

\noindent
{\bf Acknowledgements.} \\
The author would like to express his sincere gratitude to Professor Jun-ichi Segata, Faculty of Mathematics, Kyushu University, for many fruitful advices and continuous encouragement.
{This work was partly supported by Grant-in-Aid for JSPS Research Fellow Grant Number JP20J20941.}

\end{document}